\documentclass{amsart}
\usepackage{latexsym}
\usepackage{palatino}
\usepackage{amsmath, amssymb}
\usepackage{color}
\usepackage[all]{xy}
\newtheorem{dfs}{Definition}[section]
\newtheorem{lms}[dfs]{Lemma}
\newtheorem{thms}[dfs]{Theorem}
\newtheorem{props}[dfs]{Proposition}

\newtheorem{cors}[dfs]{Corollary}
\newtheorem{rems}[dfs]{Remark}

\newtheorem*{thm*}{Theorem}

\title[Comparison theory and smooth minimal C$^*$-dynamics]
{Comparison theory and smooth minimal C$^*$-dynamics}
\author{Andrew S. Toms}
\address{Department of Mathematics and Statistics, York University,
4700 Keele St., Toronto, Ontario, Canada, M3J 1P3}
\email{atoms@mathstat.yorku.ca}
\thanks{This research was supported in part by an NSERC Discovery Grant}

\begin{document}

\begin{abstract}
We prove that the C$^*$-algebra of a minimal diffeomorphism satisfies Blackadar's Fundamental Comparability Property
for positive elements.  This leads to the classification, in terms of $\mathrm{K}$-theory and traces, of the isomorphism classes of 
countably generated Hilbert modules over such algebras, and to a similar classification for the closures of 
unitary orbits of self-adjoint elements.  We also obtain a structure theorem for the Cuntz semigroup in this setting, and prove a 
conjecture of Blackadar and Handelman:  the lower semicontinuous dimension functions are weakly dense in the space of all dimension
functions.  These results continue to hold in the broader setting of unital simple ASH algebras 
with slow dimension growth and stable rank one.  Our main tool is a sharp bound on the radius of comparison of a recursive subhomogeneous 
C$^*$-algebra.  This is also used to construct uncountably many non-Morita-equivalent simple separable amenable C$^*$-algebras
with the same $\mathrm{K}$-theory and tracial state space, providing a C$^*$-algebraic analogue of McDuff's uncountable
family of $\mathrm{II}_1$ factors.  We prove in passing that the range of the radius of comparison is exhausted by simple C$^*$-algebras.
\end{abstract}  

\maketitle

\section{Introduction}

The comparison theory of projections is fundamental to the theory of von Neumann algebras, and is the basis for the
type classification of factors.  For a general C$^*$-algebra this theory is vastly more complicated, but but no less
central.  Blackadar opined in \cite{Bl4} that ``the most important general structure question concerning simple
C$^*$-algebras is the extent to which the Murray-von Neumann comparison theory for factors is valid in arbitrary simple 
C$^*$-algebras.''  In this article we answer Blackadar's question for the C$^*$-algebras associated to smooth minimal
dynamical systems, among others, and give several applications.

Tellingly, Blackadar's quote makes no mention of projections.  A C$^*$-algebra may have few or no projections, in which case
their comparison theory says little about the structure of the algebra.  The appropriate
replacement for projections is positive elements, along with a notion of comparison for the latter which generalises
Murray-von Neumann comparison for projections.  This idea was first introduced by Cuntz in \cite{Cu1} with a view to 
studying dimension functions on simple C$^*$-algebras.  His comparison relation is conveniently encoded in what is
now known as the {\it Cuntz semigroup}, a positively ordered Abelian monoid whose construction is analogous to that of the
Murray-von Neumann semigroup.  When the natural partial order on this semigroup is governed by traces, then we say
that the C$^*$-algebra has {\it strict comparison of positive elements} (see
Subsection \ref{dimfunc} for a precise definition);  this property, first introduced in \cite{Bl4}, is also known as {\it Blackadar's 
Fundamental Comparability Property for positive elements}.  It is the best available 
analogue among simple C$^*$-algebras for the comparison theory of projections in a factor, and a powerful regularity 
property necessary for the confirmation of G. A. Elliott's $\mathrm{K}$-theoretic rigidity 
conjecture (see \cite{ET} and \cite{To2}).  Its connection with the comparison theory of projections in a von Neumann algebra is 
quite explicit:  if a unital simple stably finite 
C$^*$-algebra $A$ has strict comparison of positive elements, then Cuntz comparison for those positive elements with zero in
their spectrum is synonymous with Murray-von Neumann comparison of the corresponding support projections in the bidual;  the remaining
positive elements have support projections which are contained in $A$, and Cuntz comparison for these elements reduces to Murray-von Neumann 
comparison of their support projections in $A$, as opposed to $A^{**}$.

Our main result applies to a class of C$^*$-algebras which contains properly the C$^*$-algebras associated to
minimal diffeomorphisms.  Recall that a C$^*$-algebra is {\it subhomogeneous} if
there is a uniform bound on the dimensions of its irreducible representations, and {\it approximately subhomogeneous
(ASH)} if it is the limit of a direct system of subhomogeneous C$^*$-algebras.  There are no known examples of simple
separable amenable stably finite C$^*$-algebras which are not ASH.  Every unital separable ASH algebra is the limit of
a direct sequence of {\it recursive subhomogeneous C$^*$-algebras}, a particularly tractable kind of subhomogeneous
C$^*$-algebra (\cite{NW}). 

\begin{thms}\label{master}
Let $(A_i,\phi_i)$ be a direct sequence of recursive subhomogeneous C$^*$-algebras with slow dimension growth.  
Suppose that the limit algebra $A$ is unital and simple.  It follows that $A$ has strict comparison of positive
elements.  
\end{thms}   

\noindent
We note that the hypothesis of slow dimension growth is necessary, as was shown by Villadsen in \cite{V1}.
The relationship between Theorem \ref{master} and the C$^*$-algebras of minimal dynamical systems is derived
from the following theorem:

\begin{thms}[Lin-Phillips, \cite{LP}]\label{diffash}
Let $M$ be a compact smooth connected manifold, and let $h:M \to M$ be a minimal diffeomorphism.  It follows that
the transformation group C$^*$-algebra C$^*(M,\mathbb{Z}, h)$ is a unital simple direct limit of recursive subhomogeneous
C$^*$-algebras with slow dimension growth (indeed, no dimension growth).
\end{thms}

\noindent
$\mathrm{K}$-theoretic considerations show the class of C$^*$-algebras covered by Theorem \ref{master} to be considerably
larger than the class covered by Theorem \ref{diffash}.

Let us describe briefly the applications of our main result. 
In a C$^*$-algebra $A$ of stable rank one, the Cuntz semigroup can be identified
with the semigroup of isomorphism classes of countably generated Hilbert $A$-modules---addition corresponds to the direct sum,
and the partial order is given by inclusion of modules (\cite{CEI}).  It is also 
known that positive elements $a,b \in A$ are approximately unitarily equivalent if and only if the canonical maps
from $C_0(0,1]$ into $A$ induced by $a$ and $b$ agree at the level of the Cuntz semigroup (\cite{CE}).  Thus, to the extent that one knows the
structure of the Cuntz semigroup, one also knows what the isomorphism classes of Hilbert $A$-modules and the closures
of unitary orbits of positive operators look like.  If $A$ is in addition unital, simple, exact, and has strict comparison
of positive elements, then its Cuntz semigroup can be described in terms of $\mathrm{K}$-theory and traces (see \cite[Theorem 2.6]{BT}),
and the Ciuperca-Elliott classification of orbits of positive operators extends to self-adjoint elements.
Thus, for the algebras of Theorem \ref{master}, under the additional assumption of stable rank one, we have a description of the
countably generated Hilbert $A$-modules and of the closures of unitary orbits of self-adjoints in terms of $\mathrm{K}$-theory
and traces.  (In fact, this description also captures the inclusion relation for the said modules, and the structure of their
direct sums.)  This result applies to the C$^*$-algebras of minimal diffeomorphisms as these were shown to have stable rank one by N. C.
Phillips (\cite{P5}).  Our classification is quite practical, as the $\mathrm{K}$-theory of these algebras is accessible through the 
Pimsner-Voiculescu sequence and their traces have a nice description as the invariant measures on the manifold $M$.  
The classification of Hilbert modules obtained is analogous to the classification of W$^*$-modules over a $\mathrm{II}_1$
factor.  (See \cite{BT} and Subsections \ref{hilbert} and \ref{selfadjoint}.)  Finally, we note that Jacob has recently
obtained a description of the natural metric on the space of unitary orbits of self-adjoint elements in a unital simple ASH algebra 
under certain assumptions, one of which is strict comparison.  This gives another application of Theorem \ref{master} (\cite{J}). 

It was shown in \cite[Theorem 6.4]{BPT} that if the structure theorem for the Cuntz semigroup alluded to above holds for $A$, then the
lower semicontinuous dimension functions on $A$ are weakly dense in the space of all dimension functions on $A$, confirming
a conjecture of Blackadar and Handelman from the early 1980s.  This conjecture 
therefore holds for the algebras of Theorem \ref{master}.  (See Subsections \ref{cuntz} and \ref{bh}.)

If $A$ is a unital stably finite C$^*$-algebra, then one can define a nonnegative real-valued invariant called the
{\it radius of comparison} which measures the extent to which the order structure on the Cuntz semigroup of $A$ is 
determined by (quasi-)traces.  This invariant has proved useful in the matter of distinguishing simple separable amenable
C$^*$-algebras both in general (\cite{To4}) and in the particular case of minimal C$^*$-dynamical systems (\cite{kg}). 
The proof of Theorem \ref{master} follows from a sharp upper bound that we obtain for the radius of comparison of a 
recursive subhomogeneous C$^*$-algebra.  This bound generalises and improves substantially upon our earlier bound for
homogeneous C$^*$-algebras (\cite{To5}).  In addtion to being crucial for the proof of Theorem \ref{master}, this
bound has other applications.  We use it to prove that the range of the radius of comparison is exhausted by
simple C$^*$-algebras, and that there are uncountably many non-Morita-equivalent simple separable
amenable C$^*$-algebras which all have the same $\mathrm{K}$-theory and tracial state space (Theorem \ref{anticlass}). 
This last result is proved using approximately homogenenous (AH) algebras of unbounded dimension growth, and so
may be viewed as a strong converse to the Elliott-Gong-Li classification of simple AH algebras with no dimension growth (\cite{EGL}). 
It can also be viewed as a C$^*$-algebraic analogue of McDuff's uncountable family of pairwise non-isomorphic $\mathrm{II}_1$
factors (\cite{Mc}).  (See Subsections \ref{rcsec} and \ref{anticlasssec}.)

W. Winter has recently announced a 
proof of $\mathcal{Z}$-stability for a class of C$^*$-algebras which includes unital simple direct limits of recursive 
subhomogeneous C$^*$-algebras with no dimension growth, leading to an alternative proof of Theorem \ref{master} under the
stronger hypothesis of no dimension growth.  Those working on G. A. Elliott's classification program for separable amenable
C$^*$-algebras suspect that the conditions of slow dimension growth and no dimension growth are equivalent, but
this problem remains open even for AH C$^*$-algebras.  Gong has shown that no dimension
growth and a strengthened version of slow dimension growth are equivalent for unital simple AH algebras, an
already difficult result (see \cite{G}).  

The paper is organised as follows:  Section 2 collects our basic definitions and preparatory results;  Section 3 
establishes a relative comparison theorem in the Cuntz semigroup of a
commutative C$^*$-algebra;  Section 4 applies the said comparison theorem to obtain sharp bounds on the radius
of comparison of a recursive subhomogeneous algebra;  Section 5 describes our applications in detail.

\vspace{2mm}
\noindent
{\bf Acknowledgements.}  Part of this work was carried out at the Fields Institute during its Thematic Program on Operator
Algebras in the fall of 2007.  We are grateful to that institution for its support.  We would also like to thank N. P. Brown
and N. C. Phillips for several helpful conversations.

\section{Preliminaries}\label{prelim}

\subsection{The Cuntz semigroup}\label{cuntzdef}
Let $A$ be a $C^*$-algebra, and let $\mathrm{M}_n(A)$ denote the $n \times n$ 
matrices whose entries are elements of $A$.  If $A = \mathbb{C}$, then we may simply write $\mathrm{M}_n$.

Let $\mathrm{M}_{\infty}(A)$ denote the algebraic limit of the
direct system $(\mathrm{M}_n(A),\phi_n)$, where $\phi_n:\mathrm{M}_n(A) \to \mathrm{M}_{n+1}(A)$
is given by
\[
a \mapsto \left( \begin{array}{cc} a & 0 \\ 0 & 0 \end{array} \right).
\]
Let $\mathrm{M}_{\infty}(A)_+$ (resp. $\mathrm{M}_n(A)_+$)
denote the positive elements in $\mathrm{M}_{\infty}(A)$ (resp. $\mathrm{M}_n(A)$). 
Given $a,b \in \mathrm{M}_{\infty}(A)_+$, we say that $a$ is {\it Cuntz subequivalent} to
$b$ (written $a \precsim b$) if there is a sequence $(v_n)_{n=1}^{\infty}$ of
elements of $\mathrm{M}_{\infty}(A)$ such that
\[
||v_nbv_n^*-a|| \stackrel{n \to \infty}{\longrightarrow} 0.
\]
We say that $a$ and $b$ are {\it Cuntz equivalent} (written $a \sim b$) if
$a \precsim b$ and $b \precsim a$.  This relation is an equivalence relation,
and we write $\langle a \rangle$ for the equivalence class of $a$.  The set
\[
W(A) := \mathrm{M}_{\infty}(A)_+/ \sim
\] 
becomes a positively ordered Abelian monoid when equipped with the operation
\[
\langle a \rangle + \langle b \rangle = \langle a \oplus b \rangle
\]
and the partial order
\[
\langle a \rangle \leq \langle b \rangle \Leftrightarrow a \precsim b.
\]
In the sequel, we refer to this object as the {\it Cuntz semigroup} of $A$.  
(It was originally introduce by Cuntz in \cite{Cu1}.)  The
Grothendieck enveloping group of $W(A)$ is denoted by $\mathrm{K}_0^*(A)$.

Given $a \in \mathrm{M}_{\infty}(A)_+$ and $\epsilon > 0$, we denote by 
$(a-\epsilon)_+$ the element of $C^*(a)$ corresponding (via the functional
calculus) to the function
\[
f(t) = \mathrm{max}\{0,t-\epsilon\}, \ t \in \sigma(a).
\]
(Here $\sigma(a)$ denotes the spectrum of $a$.)  
The proposition below collects some facts about Cuntz subequivalence due 
to Kirchberg and R{\o}rdam.

\begin{props}[Kirchberg-R{\o}rdam (\cite{KR}), R{\o}rdam (\cite{R4})]\label{basics}
Let $A$ be a $C^*$-algebra, and $a,b \in A_+$.  
\begin{enumerate}
\item[(i)] $(a-\epsilon)_+ \precsim a$ for every $\epsilon > 0$.
\item[(ii)] The following are equivalent:
\begin{enumerate}
\item[(a)] $a \precsim b$;
\item[(b)] for all $\epsilon > 0$, $(a-\epsilon)_+ \precsim b$;
\item[(c)] for all $\epsilon > 0$, there exists $\delta > 0$ such that $(a-\epsilon)_+ \precsim (b-\delta)_+$.
\end{enumerate}
\item[(iii)] If $\epsilon>0$ and $||a-b||<\epsilon$, then $(a-\epsilon)_+ \precsim b$.
\end{enumerate}
\end{props} 

\subsection{Dimension functions and strict comparison}\label{dimfunc}
Now suppose that $A$ is unital and stably finite, and denote by $\mathrm{QT}(A)$
the space of normalised 2-quasitraces on $A$ (v. \cite[Definition II.1.1]{BH}).
Let $S(W(A))$ denote the set of additive and order preserving maps $d:W(A) \to \mathbb{R}^+$
having the property that $d(\langle 1_A \rangle) = 1$.
Such maps are called {\it states}.  Given $\tau \in \mathrm{QT}(A)$, one may 
define a map $d_{\tau}:\mathrm{M}_{\infty}(A)_+ \to \mathbb{R}^+$ by
\begin{equation}\label{ldf}
d_{\tau}(a) = \lim_{n \to \infty} \tau(a^{1/n}).
\end{equation}
This map is lower semicontinous, and depends only on the Cuntz equivalence class
of $a$.  It moreover has the following properties:
\vspace{2mm}
\begin{enumerate}
\item[(i)] if $a \precsim b$, then $d_{\tau}(a) \leq d_{\tau}(b)$;
\item[(ii)] if $a$ and $b$ are mutually orthogonal, then $d_{\tau}(a+b) = d_{\tau}(a)+d_{\tau}(b)$;
\item[(iii)] $d_{\tau}((a-\epsilon)_+) \nearrow d_{\tau}(a)$ as $\epsilon \to 0$. 
\end{enumerate}
\vspace{2mm}
Thus, $d_{\tau}$ defines a state on $W(A)$.
Such states are called {\it lower semicontinuous dimension functions}, and the set of them 
is denoted $\mathrm{LDF}(A)$. $\mathrm{QT}(A)$ is a simplex (\cite[Theorem II.4.4]{BH}), 
and the map from $\mathrm{QT}(A)$ to $\mathrm{LDF}(A)$ defined by (\ref{ldf}) is 
bijective and affine (\cite[Theorem II.2.2]{BH}).  A {\it dimension function} on
$A$ is a state on $\mathrm{K}_0^*(A)$, assuming that the latter has been equipped
with the usual order coming from the Grothendieck map.  The set of dimension functions
is denoted $\mathrm{DF}(A)$.  $\mathrm{LDF}(A)$ is a (generally proper) face of
$\mathrm{DF}(A)$.  If $A$ has the property that $a \precsim b$ whenever $d(a) < d(b)$
for every $d \in \mathrm{LDF}(A)$, then we say that $A$ has 
{\it strict comparison of positive elements}.




\subsection{Preparatory results}
We now recall and improve upon some results that will be required in the sequel.

\begin{dfs}[cf. Definition 3.4 of \cite{To5}]\label{wellsup}
Let $X$ be a compact Hausdorff space, and let $a \in \mathrm{M}_n(\mathrm{C}(X))$ 
be positive with (lower semicontinuous)
rank function $f:X \to \mathbb{Z}^+$ taking values in $\{n_1,\ldots,n_k\}$,
$n_1 < n_2 < \cdots < n_k$.  Set 
\[
F_{i,a} := \{ x \in X | f(x) = n_i \}.
\] 

We say that $a$ is well supported if, for each $1 \leq i \leq k$, there is a
projection $p_i \in \mathrm{M}_n(\mathrm{C}(\overline{F_{i,a}}))$ such that 
\[
\lim_{r \to \infty} a(x)^{1/r} = p_i(x), \ \forall x \in F_{i,a},
\]
and $p_i(x) \leq p_j(x)$ whenever $x \in \overline{F_{i,a}} \cap \overline{F_{j,a}}$ and
$i \leq j$.
\end{dfs}

\begin{thms}[T, cf. Theorem 3.9 of \cite{To5}]\label{wellapprox}
Let $X$ be a compact Hausdorff space, and let $a \in \mathrm{M}_n(\mathrm{C}(X))_+$ and $\epsilon>0$ be given.
It follows that there is $\tilde{a} \in \mathrm{M}_n(\mathrm{C}(X))_+$ with the following properties:
\begin{enumerate}
\item[(i)] $\tilde{a} \leq a$;
\item[(ii)] $||a - \tilde{a}||< \epsilon$;
\item[(iii)] $\tilde{a}$ is well supported.
\end{enumerate}
\end{thms}

\begin{rems} {\rm
In the statement of Theorem 3.9 of \cite{To5}, $X$ is required to be a finite simplicial complex, but this is
only to ensure that some further conclusions about the approximant $\tilde{a}$ can be drawn.  The proof of
this theorem, followed verbatim, also proves Theorem \ref{wellapprox}---one simply ignores all statements which
concern the simplicial structure of $X$.  An alternative proof can be found in \cite{LP}. }
\end{rems}

For our purposes, we require a different and in some ways strengthened version of Theorem \ref{wellapprox}.  
It says that the well-supported approximant $\tilde{a}$ can be obtained as a cut-down of $a$, at the possible expense of condition (i).

\begin{lms}\label{wellcut}
Let $X$, $a$, and $\epsilon$ be as in the statement of Theorem \ref{wellapprox}.  Suppose further that $a$ has norm at most one.  It follows that there is a positive element $h$ of $\mathrm{M}_n(\mathrm{C}(X))$ of norm at most one such that the following statements hold:
\begin{enumerate}
\item[(i)] $||hah-a||<\epsilon$;
\item[(ii)] $||ha-a||<\epsilon/2$ and $||ah-a||<\epsilon/2$;
\item[(iii)] $hah$ is well-supported.
\end{enumerate}
\end{lms}

\begin{proof}
Apply Theorem \ref{wellapprox} to $a$ with the tolerance $\epsilon/4$ to obtain the approximant $\tilde{a}$.  This approximant can be described as follows (the details can be found in the proof of \ref{To5}[Theorem 3.9], which is constructive).  At every $x \in X$ there are mutually orthogonal positive elements $a_1(x),\ldots,a_k(x)$ of $\mathrm{M}_n(\mathbb{C})$ such that
\[
a(x) = a_1(x) \oplus a_2(x) \oplus \cdots \oplus a_k(x).
\]
Note that $k$ varies with $x$, and that we make no claims about the continuity of the $a_i$s.  Our approximant then has the form
\[
\tilde{a}(x) = \lambda_1 a_1(x) \oplus \lambda_2 a_2(x) \oplus \cdots \oplus \lambda_k a_k(x),
\]
where $\lambda_i \in [0,1]$.  We also have that $||a_i(x)||<\epsilon/4$ whenever $\lambda_i \neq 1$, and that there is an $\eta>0$, independent of $x$, such that the spectrum of $a_i(x)$ is contained in $[\eta,1]$ whenever $\lambda_i = 1$.

Let $f:[0,1] \to [0,1]$ be the continuous map given by
\[
f(t) = \left\{ \begin{array}{ll} t/\eta, & t \leq \eta \\ 1, & t > \eta \end{array} \right. .
\]
Set 
\[
h(x) = f(\tilde{a}(x)) = f(\lambda_1 a_1(x)) \oplus f(\lambda_2 a_2(x)) \oplus \cdots \oplus 
f(\lambda_k a_k(x)),
\]
and note that $h:X \to \mathrm{M}_n(\mathbb{C})$ is indeed a positive element of $\mathrm{M}_n(\mathrm{C}(X))$ since $\tilde{a}$ is.

Let us first verify that $||ha-a||<\epsilon/2$;  the proof that $||ah-a||<\epsilon/2$ is similar.  For every $x \in X$ we have
\[
h(x)a(x) - a(x)= \bigoplus_{i=1}^k \left( f(\lambda_i a_i(x)) a_i(x)-a_i(x) \right).
\]
If $\lambda_i=1$, then $f(\lambda_i a_i(x)) = p_i(x)$, where $p_i(x)$ is the support projection of $a_i(x)$ in $\mathrm{M}_n(\mathbb{C})$.  Thus, 
\[
f(\lambda_i a_i(x)) a_i(x) - a_i(x) = p_i(x) a_i(x) -a_i(x) = a_i(x)-a_i(x) = 0.
\]
Otherwise, $||a_i(x)||<\epsilon/4$ and $||f(\lambda_i a_i(x))|| \leq 1$, whence
\[
||f(\lambda_i a_i(x)) a_i(x) - a_i(x)|| < \epsilon/4 + \epsilon/4 = \epsilon/2.
\]
We have shown that $||f(\lambda_i a_i(x)) a_i(x) - a_i(x)|| < \epsilon/2$ for each $i \in \{1,\ldots,k\}$, so that $||ha-a||<\epsilon/2$, proving (ii).
For (i), we have
\begin{eqnarray*}
\Vert hah-a \Vert & = & \Vert hah-ha+ha-a \Vert \\
& \leq & \Vert h \Vert \cdot \Vert ah-a \Vert + \Vert ha-a \Vert \\
& < & \epsilon/2 +\epsilon/2 = \epsilon.
\end{eqnarray*}

To complete the proof, we must show that $hah$ is well-supported.  The property of being well-supported depends only on the support 
projection of $hah(x)$ as $x$ ranges over $X$.  It will thus suffice for us to show that the support projection of $hah(x)$ is the 
same as that of $\tilde{a}(x)$, since $\tilde{a}$ is well-supported.  If $\lambda_i$ is zero, then so is $f(\lambda_i a_i(x)) a_i(x) 
f(\lambda_i a_i(x))$, whence both it and $\lambda_i a_i(x)$ have the same support projection, namely, zero.  If $\lambda_i \neq 0$, 
then $f(\lambda_i a_i(x)) a_i(x) f(\lambda_i a_i(x))$ is the image of $\lambda_i a_i(x)$ under the map $t \mapsto f(t)(t/\lambda_i)f(t)$.  
This map is nonzero on $(0,1]$, and it follows that $\lambda_i a_i(x)$ and $f(\lambda_i a_i(x)) a_i(x) f(\lambda_i a_i(x))$ again have 
the same support projection.  Since these statements hold for each $i \in \{1,\ldots,k\}$, we conclude that the support projections of 
$\tilde{a}(x)$ and $hah(x)$ agree for each $x \in X$.
\end{proof}

\begin{props}[Phillips, Proposition 4.2 (1) of \cite{P4}]\label{relproj}
Let $X$ be a compact Hausdorff space of finite covering dimension $d$, and let $E \subset X$ be closed.
Let $p,q \in \mathrm{M}_n(\mathrm{C}(X))$ be projections with the property that
\[
\mathrm{rank}(q(x)) + \frac{1}{2}(d-1) \leq \mathrm{rank}(p(x)), \ \forall x \in X.
\]
Let $s_0 \in \mathrm{M}_n(\mathrm{C}(E))$ be such that $s_0^*s_0 = q|_E$ and $s_0s_0^* \leq p|_E$.  It follows
that there is $s \in \mathrm{M}_n(\mathrm{C}(X))$ such that
\[
s^*s=q, \ \ ss^* \leq p, \ \ \mathrm{and} \ \ s_0=s|_E.
\]
\end{props}

We record a corollary of Proposition \ref{relproj} for use in the sequel.

\begin{cors}\label{extend1}
Let $X$ be a compact Hausdorff space of covering dimension $d \in \mathbb{N}$, and let $E_1,\ldots,E_k$
be a cover of $X$ by closed sets.  Let $p \in \mathrm{M}_n(\mathrm{C}(X))$ and $q_i \in \mathrm{M}_n(\mathrm{C}(E_i))$
be projections of constant rank for each $i \in \{1,\ldots,k\}$.  Set $n_i = \mathrm{rank}(q_i)$, and assume that $n_1 
< n_2 < \cdots < n_k$.  Assume that $q_i(x) \leq q_j(x)$ whenever $x \in E_i \cap E_j$ and $i \leq j$.  Finally, 
suppose that $n_i-\mathrm{rank}(p) \geq (1/2)(d-1)$ for every $i$.

The following statements hold:
\begin{enumerate}
\item[(i)]  there is a partial isometry $w \in \mathrm{M}_n(\mathrm{C}(X))$ such that
$w^*w = p$ and 
\[
(w w^*)(x) \leq \bigwedge_{\{ i \ | \ x \in E_i\}} q_i(x), \ \forall x \in X;
\]
\item[(ii)] if $Y \subseteq X$ is closed, $p|_Y$ corresponds to a trivial vector bundle, and
\[
p(y) \leq \bigwedge_{\{ i \ | \ y \in E_i\}} q_i(y), \ \forall y \in Y,
\]
then $p|_Y$ can be extended to a projection $\tilde{p}$ on $X$ which also corresponds to a trivial vector bundle and satisfies
\[
\tilde{p}(x) \leq  \bigwedge_{\{ i \ | \ x \in E_i\}} q_i(x), \ \forall x \in X.
\]
\end{enumerate}
\end{cors}

\begin{proof}
{\bf (i)} The rank inequality hypothesis and the stability properties of vector bundles imply that there is a partial isometry 
$w_1 \in \mathrm{M}_n(\mathrm{C}(E_1))$ such that $w_1^*w_1 = p$ and $w_1w_1^* \leq q_1$.  Since $q_1(x) \leq q_j(x)$
whenever $x \in E_1 \cap E_j$, we have
\[  
(w_1 w_1^*)(x) \leq \bigwedge_{\{ j \ | \ x \in E_j\}} q_j(x), \ \forall x \in E_1.
\]

Suppose now that we have found a partial isometry $w_i \in \mathrm{M}_n(\mathrm{C}(E_1 \cup \cdots \cup E_i))$ such 
that $w_i^*w_i = p$ and 
\begin{equation}\label{ineq2}
(w_i w_i^*)(x) \leq \bigwedge_{\{ j \ | \ x \in E_j\}} q_j(x), \ \forall x \in E_1 \cup \cdots \cup E_i.
\end{equation}
We may now apply Proposition \ref{relproj} with $X=E_{i+1}$, $E=E_{i+1} \cap (E_1 \cup \cdots \cup E_i)$, and
$s_0 = w_i|_{E_{i+1} \cap (E_1 \cup \cdots \cup E_i)}$ to extend $w_i$ to a partial isometry $w_{i+1} \in
\mathrm{M}_n(\mathrm{C}(E_1 \cup \cdots \cup E_{i+1}))$ which satisfies (\ref{ineq2}) with $i+1$ in place of $i$.
Continuing inductively yields the desired result.

\noindent
{\bf (ii) } We will explain how to extend $p|_Y$ to $\tilde{p}$ defined on $Y \cup E_1$.  The desired result then
follows from iteration of this procedure.

The projection $p|_{Y \cap E_1}$ corresponds to a trivial vector bundle, and is subordinate to $q_1$.  Let $\tilde{q}$ 
be a projection over $Y$ which corresponds to a trivial vector bundle and has the same rank as $p$.  Since both $p|_{Y \cap E_1}$ 
and $\tilde{q}|_{Y \cap E_1}$ correspond to trivial vector bundles,
there is a partial isometry $w \in \mathrm{M}_n(\mathrm{C}(Y \cap E_1))$ such that $ww^* = p|_{Y \cap E_1} \leq q_1|_{Y \cap E_1}$ and 
$w^*w = \tilde{q}|_{Y \cap E_1}$.  We may assume that this partial isometry, viewed as an isomorphism between trivial
vector bundles, respects the decomposition of both bundles into a prescribed direct (Whitney) sum of trivial line bundles;
we moreover assume that these decompositions are the restrictions of similar decompositions for $\tilde{q}$ and $p|_Y$.

Apply Proposition \ref{relproj} with $X = E_1$, $E = Y \cap E_1$, $s_0 = w$, $q=\tilde{q}$ and $p = q_1$.  The resulting
partial isometry $s$ has the following properties:  $ss^*$ is a projection corresponding to a trivial vector bundle over $E_1$, 
$ss^*$ agrees with $p$ on $Y \cap E_1$, and upon viewing $s$ as an isomorphism of vector bundles, the image of the given 
decomposition of $\tilde{q}$ into trivial line bundles extends the similar decomposition of $p$.  It follows that the projection 
\[
\tilde{p}(x) = p(x) \vee ss^*(x), x \in Y \cup E_1,
\]
corresponds to a trivial vector bundle, and extends $p|_Y$.

\end{proof}

The proof of the next lemma is contained in the proof of \cite[Proposition 3.7]{To5}.


\begin{lms}\label{rankpreserve}
Let $X$ be a compact Hausdorff space, and let $a,b \in \mathrm{M}_n(\mathrm{C}(X))$ be positive. 
Suppose that there is a non-negative integer $k$ such that 
\[
\mathrm{rank}(a(x)) + k \leq \mathrm{rank}(b(x)), \ \forall x \in X.
\]
It follows that for every $\epsilon > 0$ there is $\delta > 0$ with the property that
\[
\mathrm{rank}((a-\epsilon)_+(x)) +k \leq \mathrm{rank}((b-\delta)_+(x)), \ \forall x \in X.
\]
\end{lms}

\section{A relative comparison result in $\mathrm{M}_n(\mathrm{C}(X))$}

The goal of this section is to prove the following Lemma.

\begin{lms}\label{main}
There is a natural number $N$ such that the following statement holds:  Let $X$ be a compact metrisable Hausdorff space of finite covering dimension 
$d$, and let $Y \subseteq X$ be closed.   Let $a,b \in \mathrm{M}_n(\mathrm{C}(X)$ be positive and, for a given tolerance 
$1>\epsilon>0$, satisfy
\begin{enumerate}
\item[(i)] $||a(x)-b(x)||< \epsilon$ for each $x \in Y$, and
\item[(ii)] $\mathrm{rank}(a(x)) + (d-1)/2 \leq \mathrm{rank}(b(x))$ for each $x \notin Y$.
\end{enumerate}

It follows that there are positive elements $c,d$ and a unitary element $u$ in $\mathrm{M}_{4n}(\mathrm{C}(X))$ whose restrictions to $Y$ are 
all equal to $\mathbf{1} \in \mathrm{M}_{4n}(\mathrm{C}(Y))$, and which, upon viewing $a$ and $b$ as elements of the upper left $n \times n$ corner
of $\mathrm{M}_{4n}(\mathrm{C}(X))$, satisfy the inequality
\[
||(duc)b(duc)^*-a|| < N\sqrt{\epsilon}.
\]
\end{lms}

The proof of Lemma \ref{main} proceeds in several steps. 

\begin{lms}\label{adjusta}
Let $X$ be a compact metrisable Hausdorff space, and let $Y$ be a closed subset of $X$.  Suppose that we
have positive elements $a,b \in \mathrm{M}_n(\mathrm{C}(X))$, a tolerance $\epsilon>0$, and a natural
number $k$ satisfying
\begin{enumerate}
\item[(i)] $||a|_Y-b|_Y||<\epsilon$, and
\item[(ii)] $\mathrm{rank}(a(x)) + k \leq \mathrm{rank}(b(x)$ for each $x \notin Y$.
\end{enumerate}
It follows that there are a positive element $\tilde{a} \in \mathrm{M}_n(\mathrm{C}(X))$ and 
open neighbourhoods $U_1 \subseteq U_2$ of $Y$ with the following properties:
\begin{enumerate}
\item[(a)] $||a-\tilde{a}||< 4\epsilon$;
\item[(b)] $\overline{U_1} \subseteq U_2$;
\item[(c)] $\tilde{a}(x) = (b(x)-2\epsilon)_+$ for every $x \in \overline{U_2} \backslash U_1$;
\item[(d)] $\mathrm{rank}(\tilde{a}(x)) + k \leq \mathrm{rank}(b(x))$ for each $x \in X \backslash U_1$.
\end{enumerate}
\end{lms}

\begin{proof}
By the continuity of $a$ and $b$ we can find open neighbourhoods $U_1 \subseteq U_2 \subseteq U_3$ of $Y$ such that
$\overline{U_1} \subseteq U_2$, $\overline{U_2} \subseteq U_3$, and $||a|_{\overline{U_3}}-b|_{\overline{U_3}}||<(3/2)\epsilon$.  
Let $f:X \to [0,1]$ be a continuous map which is equal to zero on $Y \cup (X \backslash U_3)$ and 
equal to one on $\overline{U_2} \backslash U_1$.  As a first approximation to our desired element $\tilde{a}$, we define
\[
a_1(x) = (1-f(x)) a(x) + f(x) b(x).  
\]
We then have $||a_1|_{\overline{U_3}}-b|_{\overline{U_3}}||<2 \epsilon$ and $||a-a_1||<2 \epsilon$.
Now find a continuous function $g:X \to [0,1]$ which is zero on $Y$, and equal to one on $X \backslash U_1$.
Set $\tilde{a}(x) = (a_1(x)-2\epsilon g(x))_+$.  Thus, conclusions (a) and (b) are satisfied.

For each $x \in \overline{U_2} \backslash U_1$ we have $f(x)=g(x)=1$, so that $a_1(x)=b(x)$ and $\tilde{a}(x)= (b(x)-2\epsilon)_+$.
This establishes part (c) of the conclusion.  

For part (d) of the conclusion we treat two cases.  For $x \in \overline{U_3} \backslash U_1$ we have the estimate $||a_1(x)-a(x)||<2\epsilon$
and the fact that $\tilde{a}(x)=(a_1(x)-2\epsilon)_+$.  Proposition \ref{basics} (iii) then implies that $\tilde{a}(x) \precsim a(x)$,
whence 
\[
\mathrm{rank}(\tilde{a}(x)) \leq \mathrm{rank}(a(x)) \leq \mathrm{rank}(b(x)) - k. 
\]
For $x \in X \backslash U_3$ we have $a_1(x)=a(x)$ and $\tilde{a}(x) = (a_1(x)-2\epsilon)_+$.  Thus, $\tilde{a}(x) \precsim a(x)$ and we
proceed as before.  
\end{proof}

We can now make our first reduction.

\begin{lms}\label{reduction1}
In order to prove Lemma \ref{main}, it will suffice to prove the following statement, hereafter referred to as {\bf (S)}:
Let $X$ be a compact metrisable Hausdorff space of covering dimension $d \in \mathbb{N}$, and let
$Y \subseteq X$ be closed.  Let $1>\epsilon > 0$ be given.  Suppose that $\hat{a},\hat{b} \in \mathrm{M}_n(\mathrm{C}(X))_+$
have the following properties:
\begin{enumerate}
\item[(i)] $||(\hat{a}-\hat{b})|_{\overline{U}}||< \epsilon$ for some open set $U \supseteq Y$;
\item[(ii)] $\hat{b}|_{X \backslash U}$ is well-supported;
\item[(iii)] there are an open set $V \supseteq \overline{U}$ and $\gamma>0$ such that
\[
\hat{a}(x) = (\hat{b}(x)-\gamma)_+, \ \forall x \in \overline{V} \backslash U;
\]
\item[(iv)] 
\[
\mathrm{rank}(\hat{a}(x)) + (d-1)/2 \leq \mathrm{rank}(\hat{b}(x)), \ \forall x \in X \backslash U.
\]
\end{enumerate}
It follows that there are positive elements $\hat{c},\hat{d}$ and a unitary element $v$ in $\mathrm{M}_{4n}(\mathrm{C}(X))$ whose restrictions to 
$\overline{U}$ are all equal to $\mathbf{1} \in \mathrm{M}_{4n}(\mathrm{C}(\overline{U}))$, and which, upon viewing $\hat{a}$ and $\hat{b}$ as elements of 
the upper left $n \times n$ corner of $\mathrm{M}_{4n}(\mathrm{C}(X))$, satisfy the inequality
\[
||(\hat{d}v\hat{c})\hat{b}(\hat{d}v\hat{c})^*-\hat{a}|| < 4\sqrt{\epsilon}.
\]
\end{lms}

\begin{proof}

Let $a$ and $b$ be as in the hypotheses of Lemma \ref{main}.  One can immediately find an open set $U \supseteq Y$
such that $||a(x)-b(x)|| \leq \epsilon_0 < \epsilon$ for every $x \in U$.  By Lemma \ref{rankpreserve}, there is a
$\delta>0$ such that 
\[
\mathrm{rank}(a(x)-\epsilon)_+) + (d-1)/2 \leq \mathrm{rank}(b(x)-\delta)_+, \ \forall x \in X.
\]
Set $\eta = \mathrm{min} \{\epsilon-\epsilon_0, \delta\}$.

Fix an open set $W \supseteq Y$ such that $\overline{W} \subseteq U$.  Apply Lemma \ref{wellcut} to $b|_{X \backslash W}$ with the
tolerance $\eta$ to produce a positive element
$\hat{h} \in \mathrm{M}_n(\mathrm{C}(X \backslash W))$ with the properties listed in the conclusion of that lemma.  Fix a continuous map $f:X \to [0,1]$
which is equal to one on $\overline{W}$ and equal to zero on $X \backslash U$.  Set
\[
h(x) = f(x) \mathbf{1}_{\mathrm{M}_n} + (1-f(x)) \hat{h}(x), \ \forall x \in X,
\]
and $\hat{b}(x) = h(x)b(x)h(x)$.  For each $x \in X \backslash U$, we have $f(x)=0$.  It follows that $\hat{b}|_{X \backslash U} = 
(\hat{h}|_{X \backslash U})(b|_{X \backslash U})(\hat{h}|_{X \backslash U})$, whence, by part (i) of the conclusion of Lemma \ref{wellcut}, 
$\hat{b}|_{X \backslash U}$ is well-supported.  

We have
\begin{eqnarray*}
& & ||hbh-b|| \\
& = & \underset{x \in X}{\sup} ||h(x)b(x)h(x) - b(x)|| \\
& = & \underset{x \in X}{\sup} || [f(x)\mathbf{1}+ (1-f(x))\hat{h}(x))]b(x)[(f(x)\mathbf{1}+ (1-f(x))\hat{h}(x))]-b(x)|| < \eta,
\end{eqnarray*}
where the last inequality follows from part (ii) of the conclusion of Lemma \ref{wellcut}.  Since $\eta \leq \epsilon-\epsilon_0$, 
we have $||a|_{\overline{U}}-\hat{b}|_{\overline{U}}||< \epsilon$.  The inequality $\eta \leq \delta$ implies that $||\hat{b}-b||<\delta$.
Combining this fact with part (iii) of Proposition \ref{basics} yields
\[
\mathrm{rank}((a(x)-\epsilon)_+) + (d-1)/2 \leq \mathrm{rank}(b(x)-\delta)_+) \leq \mathrm{rank}(\hat{b}(x)), \ \forall x \in X.
\]

We will now apply Lemma \ref{adjusta} with $\hat{b}$, $(a-\epsilon)_+$, and $2\epsilon$ substituted for $b$, $a$, and $\epsilon$, respectively.
Note that by shrinking $U$ and $W$ above, we may assume that they will serve as the sets $U_2$ and $U_1$ of Lemma \ref{adjusta}, respectively.  Form
the approximant $\tilde{a}$ to $(a-\epsilon)_+$ provided in the conclusion of Lemma \ref{adjusta}, and set $\hat{a}=\tilde{a}$.  Note
that $\Vert \hat{a}- (a-\epsilon)_+ \Vert < 8 \epsilon$.  We have
\begin{eqnarray*}
||(\hat{a}-\hat{b})|_{\overline{U}}|| & \leq & ||(\hat{a} - (a-\epsilon)_+)|_{\overline{U}}|| + ||((a-\epsilon)_+-\hat{b})|_{\overline{U}}|| \\
& < & 2(4 \epsilon) + 2 \epsilon = 10 \epsilon
\end{eqnarray*}
and
\[
\tilde{a}(x) = (\hat{b}(x)-4\epsilon)_+, \ \forall x \in X.
\]  

Our $\hat{a}$ and $\hat{b}$ now satisfy the hypotheses of statement {\bf (S)} with $10 \epsilon$ and $4 \epsilon$ substituted for $\epsilon$ and $\gamma$,
respectively.  Let $\hat{c}$, $\hat{d}$, and $v$ be as in the conclusion of statement {\bf (S)}.  Set $u=v$, $d=\hat{d}$, and $c = \hat{c}h$.  It follows that
\begin{eqnarray*}
||(duc)b(duc)^*-a|| & = & ||(\hat{d}v\hat{c})(hbh)(\hat{d}v\hat{c})^*-a|| \\
& \leq & ||((\hat{d}v\hat{c})\hat{b}(\hat{d}v\hat{c})^* - \hat{a}|| + ||\hat{a}-a|| \\
& < & 40 \sqrt{\epsilon} \Vert + \Vert \hat{a} - (a-\epsilon)_+ \Vert + \Vert a - (a-\epsilon)_+ \Vert \\
& < & 40 \sqrt{\epsilon} + 9\epsilon < 49 \sqrt{\epsilon}.
\end{eqnarray*} 
This shows that if {\bf (S)} holds, then Lemma \ref{main} holds (with $N=49$).
\end{proof}

The next Lemma constructs the unitary $u$ of Lemma \ref{main}.

\begin{lms}\label{partialiso}
Let $X$ be a compact metrisable Hausdorff space of covering dimension $d \in \mathbb{N}$, and let $a,b \in \mathrm{M}_n(\mathrm{C}(X))$
be well-supported positive elements with the property that 
\[
\mathrm{rank}(a(x)) + \frac{1}{2}(d-1) \leq \mathrm{rank}((b-\epsilon)_+(x))
\]
for some $\epsilon>0$ and every $x \in X$.  Suppose further that $a(y) \leq (b(y)-\epsilon)_+$ for each $y$ in the closure of an open subset $Y$ of $X$, 
and that $a$ and $b$ have norm at most one.

For each $k \in \{0,1,\ldots,n\}$, set
\[
E_k = \{x \in Z \ | \ \mathrm{rank}(a(x))=k\}; \ \ F_k = \{x \in Z \ | \ \mathrm{rank}(b(x))=k\}.
\]
For each $x \in E_k$, let $p_k(x)$ be the support projection of $a(x)$;  for each $x \in F_k$ let $q_k(x)$ be the 
support projection of $b(x)$.  Since $a$ and $b$ are well-supported, the continuous projection-valued
maps $x \mapsto p_k(x)$ and $x \mapsto q_k(x)$ can be extended to $\overline{E_k}$ and $\overline{F_k}$, respectively.  We also
denote these extended maps by $p_k$ and $q_k$.

View $\mathrm{M}_n(\mathrm{C}(X))$ as the upper-left $n \times n$ corner of $\mathrm{M}_{4n}(\mathrm{C}(X))$, and let $Z \subseteq Y$ be closed.
It follows that there is a unitary $u \in \mathrm{M}_{4n}(\mathrm{C}(X))$ with the following properties:
\begin{enumerate}
\item[(i)] $u(z) = \mathbf{1}_{4n} \in \mathrm{M}_{4n}(\mathbb{C})$ for each $z \in Z$;
\item[(ii)] 
\[
(u^*p_ku)(x) \leq \bigwedge_{\{j \ | \ x \in \overline{F_j}\} } q_j(x), \ \forall x \in E_k, \ \forall k \in \{0,\ldots,n\};
\]
\item[(iii)] $u$ is homotopic to the unit of $\mathrm{M}_{4n}(\mathrm{C}(X))$.
\end{enumerate} 
\end{lms}

\begin{proof}

\noindent
{\bf Step 1.} For each $y \in \overline{Y}$, set $v(y) = \mathbf{1}_n$.   Let us verify conclusion (ii) above
with $v$ in place of $u$ for each $x \in \overline{Y} \cap \overline{E_k}$.  For each $y \in \overline{Y}$ we have $a(y) \leq b(y)$, and so
\[
a(y)^{1/2^n} \leq b(y)^{1/2^n}, \ \forall n \in \mathbb{N}.
\]
It follows that $p_k(y) \leq q_j(y)$ for each $y \in \overline{Y} \cap E_k \cap \overline{F_j}$, and so
\[
(v^*p_kv)(y) = p_k(y) \leq \bigwedge_{\{j \ | \ y \in \overline{F_j} \} } q_j(y)
\]
for each $y \in \overline{Y} \cap E_k$ and $k \in \{1,\ldots,n\}$. 
It remains to prove that the inequality above holds when $y \in \overline{Y} \cap \overline{E_k}$.  

Set $r(x) = \chi_{[\epsilon/2,1]}(b(x))$ for each $x \in X$, so that $r(x)$ dominates the support projection of $(b-\epsilon)_+$ at 
$x$.  It follows that $p_k(x) \leq r(x)$ for each $x \in E_k$.  In fact, 
\begin{equation}\label{ineq1}
p_k(x) \leq r(x) \leq  \bigwedge_{\{j \ | \ x \in \overline{F_j} \} } q_j(x)
\end{equation}
for each $k \in \{0,\ldots,n\}$ and $x \in E_k$, where the second inequality follows from the fact that 
\[
 \bigwedge_{\{j \ | \ x \in \overline{F_j} \} } q_j(x)
\]
is the support projection of $b(x)$ for each $x \in E_k$.  It will suffice
to prove that the first inequality holds for $y \in \overline{Y} \cap \overline{E_k}$.
It is well known that $r(x)$ is an upper semicontinuous projection-valued
map from $X$ into $\mathrm{M}_n(\mathrm{C}(X))$.  Fix $y \in \overline{Y} \cap \overline{E_k}$, and let $(y_n)$ be a sequence in 
$\overline{Y} \cap E_k$ converging to $y$.  Since $p_k(y_n) \leq r(y_n)$ for each $n \in \mathbb{N}$ we have
\[
(p_k(y_n) \xi | \xi ) \leq (r(y_n)\xi|\xi), \ \forall \xi \in \mathbb{C}^n, \ n \in \mathbb{N}.
\]
Now
\[
(p_k(y) \xi | \xi) = \lim_{n \to \infty} (p_k(y_n) \xi | \xi) \leq \limsup_{n \to \infty} (r(y_n) \xi | \xi) \leq (r(y) \xi | \xi), \ \forall \xi \in
\mathbb{C}^n,
\]
where the last inequaltiy follows from the upper semicontinuity of $r$.  It follows that $p_k(y) \leq r(y)$, as required.
 
 \vspace{3mm}
\noindent
{\bf Step 2.} We will construct partial isometries 
$v_k \in \mathrm{M}_n(\mathrm{C}(\overline{E_k} \backslash Y))$ with the following properties:
\begin{enumerate}
\item[(a)] 
\[
(v_k^*p_kv_k)(x) \leq \bigwedge_{\{j \ | \ x \in \overline{F_j} \backslash Y\} } q_j(x), \ \forall x \in \overline{E_k} \backslash Y, 
\ \forall k \in \{1,\ldots,n\};
\]
\item[(b)] the $v_k$s are compatible in the sense that for each $x \in \overline{E_i} \cap \overline{E_j} \backslash Y$ with $i \leq j$,
\[
(v_i^*p_iv_i)(x) = (v_j^*p_iv_j)(x);
\]
\item[(c)] for each $x \in \overline{E_k} \cap \partial Y$, $v_k(x) = p_k(x) = v(x) p_k(x)$.
\end{enumerate}
In the third step of the proof, we will extend the $v$ from Step 1 and the $v_k$s above to produce the unitary $u$ required by the lemma.

We will prove the existence of the required $v_k$s by induction on the number of rank values taken by $a$.
Let us first address the case where $a$ has constant rank equal to $k_0$.  In this case $E_{k_0} = \overline{E_{k_0}} = X$, and $a$ is Cuntz equivalent
to the projection $p_{k_0} \in \mathrm{M}_n(\mathrm{C}(X))$.   We set $v_{k_0}(y) = p_{k_0}(y)$ for each $y \in \partial Y$, thus satisfying
requirements (a) and (c) for these $y$.  (Note that condition (b) is met trivially in the present case.)  Let $j_1 < j_2 < \cdots < j_l$ be the indices 
for which $\overline{F_{j_i}} \neq \emptyset$.  The existence of the required partial isometry extending the definition of $v_{k_0}$ on $\partial Y$ now
follows from repeated application of Proposition \ref{relproj}:  one substitutes $p_{k_0}$ and $q_{j_i}$ for $q$ and $p$, respectively, in the hypotheses
of the said Proposition.

Now let us suppose that we have found partial isometries $v_0,\ldots,v_k$ satisfying (a), (b), and (c) above.  We must construct $v_{k+1}$, assuming $k<n$.
We will first construct $v_{k+1}$ on the boundary 
\[
\overline{E_{k+1}} \cap (\overline{E_1} \cup \overline{E_2} \cup \cdots \cup \overline{E_k}) \cap Y^c.
\]
For $x \in \overline{E_{k+1}} \cap \overline{E_k} \cap Y^c$, we have
\[
(v_k^*p_kv_k)(x) \leq \bigwedge_{\{j \ | \ x \in \overline{F_j} \backslash Y\} } q_j(x).
\]
From (\ref{ineq1}) on $\overline{E_{k+1}}$ we also have that the rank of the right-hand side exceeds that of the
left hand side by at least
\[
\mathrm{rank}(p_{k+1}(x)-p_k(x)) + \frac{1}{2}(d-1).
\]
Working over $\overline{E_{k+1}} \cap \overline{E_k} \cap Y^c$, we have that $(p_{k+1}-p_k)$ is Murray-von Neumann equilvalent
 to a projection $f_{k}$ which is orthogonal to $v_k^*p_kv_k$ and satisfies
\[
f_{k}(x) \leq \bigwedge_{\{j \ | \ x \in \overline{F_j} \backslash Y\} } q_j(x).
\]
(This follows from part (i) of Corollary \ref{extend1}.)
Let $w_k$ be a partial isometry defined over $\overline{E_{k+1}} \cap \overline{E_k} \cap Y^c$ such that $w_k^*(p_{k+1}-p_k)w_k = f_{k}$,
and set $v_{k+1} = v_k + w_k$.  With this definition we have
\[
(v_{k+1}^*p_{k+1}v_{k+1})(x) \leq \bigwedge_{\{j \ | \ x \in \overline{F_j} \backslash Y\} } q_j(x), 
\]
and
\[
(v_k^* p_k v_k)(x) = (v_{k+1}^*p_{k}v_{k+1})(x) 
\]
for each $x \in \overline{E_{k+1}} \cap \overline{E_k} \cap Y^c$.

Let us now show how to extend $v_{k+1}$ one step further, to $E_{k+1} \cap (\overline{E_k} \cup \overline{E_{k-1}}) \cap Y^c$;
its successive extensions to the various 
\[
E_{k+1} \cap (\overline{E_k} \cup \cdots \cup E_{k-j}) \cap Y^c, \ j \in \{1,\ldots,k-1\},
\]
are similar, and the details are omitted.

In this paragraph we work over the set $\overline{E_{k+1}} \cap (\overline{E_{k}} \cup \overline{E_{k-1}}) \cap Y^c$.  We will suppose that this set
contains $\overline{E_{k+1}} \cap \overline{E_k} \cap Y^c$ strictly, for there is otherwise no extension of $v_{k+1}$ to be made.  Over 
$(\overline{E_{k+1}} \cap \overline{E_k} \cap Y^c) \cap E_{k-1}$, we set $w_{k-1} = v_{k+1} (p_{k+1}-p_k)$.  Thus, $w_{k-1}$ is a partial isometry
carrying (the restriction of) $p_{k+1}-p_{k-1}$ to a subprojection of 
\[
Q(x) \stackrel{\mathrm{def}}{=} \left( \bigwedge_{\{j \ | \ x \in \overline{F_j} \backslash Y\} } q_j(x) \right) - (v_{k-1} p_{k-1} v_{k-1}^*)(x), 
\ x \in (\overline{E_{k+1}} \cap \overline{E_k} \cap Y^c) \cap E_{k-1}.
\]
We moreover have the rank inequality
\[
[\mathrm{rank}(Q(x)) - \mathrm{rank}((v_{k-1} p_{k-1} v_{k-1}^*)(x))] - \mathrm{rank}((w_{k-1} w_{k-1}^*)(x)) \geq \frac{1}{2}(d-1).
\]
Applying part (i) of Corollary \ref{extend1}, we extend $w_{k-1}$ to a partial isometry defined on all of $\overline{E_{k+1}} \cap 
\overline{E_{k-1}} \cap Y^c$ which has the property that $(w_{k-1} w_{k-1}^*) \leq Q(x)$.  Finally, set $v_{k+1} = v_{k-1} + w_{k-1}$
on this set.  It is straightforward to check that $v_{k+1}$ has the required properties.

Iterating the arguments above, we have an appropriate definition of $v_{k+1}$ on 
\[
\overline{E_{k+1}} \cap (\overline{E_1} \cup \overline{E_2} \cup \cdots \cup \overline{E_k}) \cap Y^c.
\]
To extend the definition of $v_{k+1}$ from the set above to all of $\overline{E_{k+1}} \cap Y^c$, we simply
apply part (i) of Corollary \ref{extend1}.

\vspace{3mm}
\noindent
{\bf Step 3.}  Set $H_{-1} = \overline{Y}$ and $H_k = E_k \backslash Y$, so that $H_{-1},\ldots,H_n$ is a closed cover of $X$.  
For each $k \in \{-1,0,\ldots,n\}$ we
have a partial isometry $v_k \in \mathrm{M}_n(\mathrm{C}(H_k))$ from Steps 1 and 2 (assuming that $v_{-1} = v = \mathbf{1}$).  
Let $r_k$ denote the source projection
of $v_k$.  Notice that $r_k$ agrees with $p_k$ off $\overline{Y}$.  In this final step of our proof, we will construct the required unitary $u$ in a 
manner which extends the $v_k$:  $(u|_{H_k}) r_k = v_k$.

Suppose that we have found a partial isometry $w_k \in \mathrm{M}_{2n}(\mathrm{C}(H_{-1} \cup \cdots \cup H_k))$ with source projection equal to 
$\mathbf{1}_n$ (i.e., the unit of the upper-left $n \times n$ corner)
and satisfying $(w_k|_{H_j}) r_j = v_j$ for each $j \in \{0,\ldots,k\}$.  Let us show that $k$ can be replaced with $k+1$, and that $w_{k+1}$ may 
moreover be chosen to be an extension of $w_k$. 

Over $(H_{-1} \cup \cdots \cup H_k) \cap H_{k+1}$, $w_k$ carries the projection $\mathbf{1}_n - r_{k+1}$ into a subprojection of $\mathbf{1}_{2n} -
v_{k+1}v_{k+1}^*$.  The rank of the latter projection exceeds that of the former by at least $(d-1)/2$, and so the partial isometry 
$w_k (\mathbf{1}_n-r_{k+1})$ defined over $(H_{-1}\cup \cdots \cup H_k) \cap H_{k+1}$ can be extended to a partial isometry $w^{'}_{k+1}$ defined 
over $H_{k+1}$ which carries $\mathbf{1}_n - r_{k+1}$ into a subprojection of $\mathbf{1}_{2n} - v_{k+1}v_{k+1}^*$ 
(cf. Proposition \ref{relproj}).  Setting $w_{k+1} = 
v_{k+1} + w^{'}_{k+1}$ on $H_{k+1}$ and $w_{k+1} = w_k$ otherwise gives the desired extension.  Iterating this extension process yields a 
partial isometry $w \in \mathrm{M}_{2n}(\mathrm{C}(X))$ with source projection $\mathbf{1}_n$ satisfying $(w|_{H_k}) r_k = v_k$. 

To complete the proof, it will suffice to find a unitary $u \in \mathrm{M}_{4n}(\mathrm{C}(X))$ which is homotopic to the identity (for conclusion (iii)), satisfies $u \mathbf{1}_n = w$ (for conclusion (ii)), and is equal to $\mathbf{1} \in \mathrm{M}_{4n}(\mathbb{C})$ over $Z$ (for conclusion (i)).  
We will find a unitary $s$ satisfying (ii) and (iii), and then modify it to obtain $u$.  

The complement of $\mathbf{1}_n$ in $\mathrm{M}_{2n}(\mathrm{C}(X))$ is Murray-von Neumann
equivalent to the complement of $ww^*$, as both projections have the same $\mathrm{K}_0$-class and are of rank at least $(d-1)/2$.  Let $w^{'}$ be a partial isometry implementing this equivalence.  It follows that $w + w^{'} \in \mathrm{M}_{2n}(\mathrm{C}(X))$ is unitary.  Setting $s = (w+w^{'}) \oplus
(w+w^{'})^*$ yields our precursor to the required unitary $u \in \mathrm{M}_{4n}(\mathrm{C}(X))$---the 
$\mathrm{K}_1$-class of $s$ is zero, so it is homotopic to $\mathbf{1}_{4n}$ by virtue of its rank (\cite[Theorem 10.12]{Ri1}).  
The unitary $s|_{\overline{Y}} \in \mathrm{M}_{4n}(\mathrm{C}(\overline{Y}))$ has the form 
$\mathbf{1}_n \oplus \tilde{s}$, where $\tilde{s}$ is  a $3n \times 3n$ unitary homotopic to the identity.  (This follows
from two facts:  the $\mathrm{K}_1$-class of $\mathbf{1}_n \oplus \tilde{s}$ is zero, and the natural map
$\iota: \mathcal{U}(\mathrm{M}_{3n}(\mathrm{C}(\overline{Y})) \to  \mathcal{U}(\mathrm{M}_{4n}(\mathrm{C}(\overline{Y}))$ given
by $x \mapsto \mathbf{1}_n \oplus x$ is injective by \cite[Theorem 10.12]{Ri1}.) Let
\[
H:\overline{Y} \times [0,1] \to \mathcal{U}(\mathrm{M}_{3n}(\mathbb{C}))
\]
be a homotopy such that $H(y,0) = \tilde{s}(y)$ and $H(y,1) = \mathbf{1}_{3n} \in \mathrm{M}_{3n}(\mathbb{C})$.  
Let $h:\overline{Y} \to [0,1]$ be a continuous map equal to one on Z and equal to 
zero on $\partial Y$.  Finally, define
\[
u(x) = \left\{ \begin{array}{ll} s(x), & x \notin \overline{Y} \\ \mathbf{1}_n \oplus H(x,f(x)), & x 
\in \overline{Y} \end{array}  \right. .
\]
The unitary $u$ is clearly homotopic to $s$, and so satisfies conclusion (iii).  Conclusion (i) holds for $u$ by construction, and conclusion (ii) holds since $u \mathbf{1}_n = s \mathbf{1}_n = w$.
\end{proof}

\begin{lms}\label{extendcomp}
The statement {\bf (S)} (cf. Lemma \ref{reduction1}) holds. 
\end{lms}

\begin{proof}  {\bf Step 1.}
To avoid cumbersome notation, we use $a$, $b$, $c$, and $d$ in place of their ``hatted'' versions in the hypotheses and conclusion of {\bf (S)}.  We will
first find the unitary $v$ and the positive elements $c$ and $d$ required by the conclusion of {\bf (S)} with two failings:  $c$ and $d$ are not necessarily  
equal to $\mathbf{1} \in \mathbb{M}_{4n}(\mathbb{C})$ at each point of $\overline{U}$, and the estimate
\[
\left| \left| (cvd)b(cvd)^* - a \right| \right| <  4\sqrt{\epsilon}
\]
only holds on $X \backslash U$.  Both of these failings will be attributable to $c$ and $d$ alone, and will be repaired in later Steps 2. and 3.

By combining the hypotheses (i) and (iii) of {\bf (S)}, we may, after perhaps shrinking the set $V$, assume that $\gamma <\epsilon$.  With
this choice of $V$ we also have that hypothesis (i) holds with $\overline{V}$ in place of $\overline{U}$.  We will also weaken hypothesis
(iii) to an inequality.  This has two advantages.  First, by replacing $a$ with $(a-\delta)_+$ for some small $\delta > 0$, we can assume that (iv) holds with $b$ replaced by $(b-\eta)_+$ for some $\gamma > \eta >0$.  Second, we can assume that $a|_{X \backslash U}$ is well-supported by using the following procedure:  let $W \supseteq Y$ be an open set whose closure is contained in $U$;  replace $a$ with a suitably close approximant $\tilde{a}$ on $X \backslash W$, as provided by Lemma \ref{wellcut};  choose a continuous map $f:X \to [0,1]$ which is equal to one on $\overline{W}$ and equal to zero on $X \backslash U$;  replace the original $a$ with the positive element equal to $f(x)a(x) +
(1-f(x))\tilde{a}(x)$ at each $x \in X$.  Let us summarise our assumptions:
\begin{enumerate}
\item[(i)] $||(a-b)|_{\overline{V}}||< \epsilon$ for some open set $V \supseteq Y$;
\item[(ii)] $b|_{X \backslash U}$ and $a|_{X \backslash U}$ are well-supported (and $\overline{U} \subseteq V$);
\item[(iii)] there is $0 < \gamma < \epsilon$ such that
\[
a(x) \leq (b(x)-\eta)_+, \ \forall x \in \overline{V} \backslash U;
\]
\item[(iv)] 
\[
\mathrm{rank}(a(x)) + (d-1)/2 \leq \mathrm{rank}((b-\eta)_+(x)), \ \forall x \in X \backslash U.
\]
\end{enumerate}
Set $Z = X \backslash U$ and $W = V \backslash U$.

For each $k \in \{0,1,\ldots,n\}$, set
\[
E_k = \{x \in Z \ | \ \mathrm{rank}(a(x))=k\}; \ \ F_k = \{x \in Z \ | \ \mathrm{rank}(b(x))=k\}.
\]
For each $x \in E_k$, let $p_k(x)$ be the support projection of $a(x)$.  Similarly, define $q_k(x)$ to be the 
support projection of $b(x)$ for each $x \in F_k$.  Since (the restrictions of) $a$ and $b$ are well-supported on $Z$, the continuous projection-valued
maps $x \mapsto p_k(x)$ and $x \mapsto q_k(x)$ can be extended to $\overline{E_k}$ and $\overline{F_k}$, respectively.  We also
denote these extended maps by $p_k$ and $q_k$.  Let $\widetilde{V_1}$ be an open subset of $X$ such that $\overline{U} \subseteq
\widetilde{V_1} \subseteq \overline{\widetilde{V_1}} \subseteq V$, and set $V_1 = \widetilde{V_1} \cap Z$.  
Apply Lemma \ref{partialiso} with $b|_{Z}$, $a|_{Z}$, $Z$,
$W$, $\overline{V_1}$, and $\eta$ substituted for the variables $b$, $a$, $X$, $Y$, $Z$, and $\epsilon$ in the hypotheses of the lemma, 
respectively.  Let $u$ be the unitary in $\mathrm{M}_{4n}(\mathrm{C}(Z))$
provided by the conclusion of the said lemma.  Define $v \in \mathrm{M}_{4n}(\mathrm{C}(X))$ to be the unitary which is equal to $u$ on
$Z$ and equal to $\mathbf{1} \in \mathrm{M}_{4n}(\mathbb{C})$ at each point of $U$.  This $v$ will serve as the unitary required in 
the conclusion of {\bf (S)}.  We will simply use $v$ in place of $v|_Z$ whenever it is clear that we are working over $Z$.

From conclusion (ii) of Lemma \ref{partialiso} we have
\begin{equation}\label{comp1}
p_k(x) \leq v(x)\left[ \bigwedge_{\{j \ | \ x \in \overline{F_j}\} } q_j(x) \right]v(x)^*, \ \forall x \in E_k, \ \forall k \in \{0,\ldots,n\}.
\end{equation}
For each $\delta>0$ let $f_\delta, g_\delta:[0,1] \to [0,1]$ be given by the formulas
\[
f_\delta(t) = \left\{ \begin{array}{ll} 0, & t \in [0,\delta/2] \\ (2t-\delta)/\delta, & t \in (\delta/2,\delta) \\
1, & t \in [\delta,1] \end{array} \right. ,
\]
and 
\[
g_\delta(t) = \left\{ \begin{array}{ll} 0, & t \in [0,\delta/2] \\ f_\delta(t)/t, & t \in (\delta/2,1] \end{array} \right. .
\]
Note that $f_\delta(t)$ and $g_\delta(t)$ are continuous, and that $tg_\delta(t)  = f_\delta(t)$.

Consider the following product in $\mathrm{M}_{4n}(\mathrm{C}(Z))$:
\begin{equation}\label{prod1}
(\sqrt{a} v \sqrt{g_\delta(b)}) b(\sqrt{a} v \sqrt{g_\delta(b)})^* = (\sqrt{a} v \sqrt{g_\delta(b)}) b (\sqrt{g_\delta(b)} v^* \sqrt{a}).
\end{equation}
As $\delta \to 0$ we have 
\[
[\sqrt{g_\delta(b)} b \sqrt{g_\delta(b)}](x) = f_\delta(b)(x) \to \bigwedge_{\{j \ | \ x \in \overline{F_j}\} } q_j(x), \ \forall x \in Z.
\] 
Thus, by (\ref{comp1}), $[v \sqrt{g_\delta(b)} b \sqrt{g_\delta(b)} v^*](x)$ converges to a projection which dominates the support projection 
of $a(x)$.  It follows that the product (\ref{prod1}), evaluated at $x \in Z$, converges to $a(x)$ as $\delta \to 0$.  We will prove that this
convergence is uniform in norm on $Z$.

If $\delta < \kappa$, then $f_\delta(b) \geq f_\kappa(b)$.  It follows that
\begin{equation}\label{ineq3}
\sqrt{a} v f_\delta(b) v^* \sqrt{a} \geq \sqrt{a} v f_\kappa(b) v^* \sqrt{a}.
\end{equation}
Since $b \leq \mathbf{1}$, we have 
\[
\sqrt{a} v f_\delta(b) v^* \sqrt{a} \leq \sqrt{a} vv^* \sqrt{a} = a,
\]
and similarly for $f_\kappa(b)$.  Combining this with (\ref{ineq3}) yields
\[
0 \leq a-\sqrt{a} v f_\delta(b) v^* \sqrt{a} \leq a-\sqrt{a} v f_\kappa(b) v^* \sqrt{a}.
\]
By positivity,
\begin{equation}\label{ineq4}
\left| \left| a-\sqrt{a} v f_\delta(b) v^* \sqrt{a} \right| \right| \leq \left| \left| a-\sqrt{a} v f_\kappa(b) v^* \sqrt{a} \right| \right|.
\end{equation}
Let $(\delta_n)$ be a sequence of strictly positive tolerances converging to zero.  By
(\ref{ineq4}), 
\[
\left| \left| [a-\sqrt{a} v \sqrt{g_{\delta_n}(b)} b \sqrt{g_{\delta_n}(b)} v^* \sqrt{a}](x) \right| \right| = 
\left| \left| [a-\sqrt{a} v f_{\delta_n}(b) v^* \sqrt{a}](x) \right| \right|
\]
is a monotone decreasing sequence converging to zero for each $x \in Z$.  By Dini's Theorem,
this sequence converges uniformly to zero on $Z$.  For the remainder of the proof we fix $\epsilon > \delta>0$ with
the property that
\begin{equation}\label{est1}
\left| \left| a-\sqrt{a} v \sqrt{g_{\delta}(b)} b \sqrt{g_{\delta}(b)} v^* \sqrt{a} \right| \right| < \epsilon.
\end{equation}
Extend $\sqrt{a}$ and $\sqrt{g_\delta(b)}$ to positive elements $c$ and $d$ in $\mathrm{M}_{4n}(\mathrm{C}(X))$, respectively.  This choice
of $c$ and $d$ completes Step 1.

\vspace{3mm}
\noindent
{\bf Step 2.}  We must now modify our choice of $c$ and $d$ to address their failings, outlined at the beginning of Step 1. 
This modification will be made in three smaller steps.  In a slight abuse of notation, we will use $c$ and $d$ to denote the
successive modifications of the present $c$ and $d$.

For each $x \in \overline{W}$ we have $b(x)-a(x) \geq 0$ and $||b(x)-a(x)|| < \epsilon$.  
It is a straightforward exercise to show that
\[
\left| \left| \sqrt{b(x)} - \sqrt{a(x)} \right| \right| < \sqrt{\epsilon}.
\]
Choose a continuous
map $f_1:Z \to [0,1]$ which is equal to one on $Z \backslash W$ and equal to zero on $\overline{V_1}$.  Set $a_1(x) = f_1(x)\sqrt{a(x)} + 
(1-f_1(x))\sqrt{b(x)}$ for each $x \in Z$, and set $s = v \sqrt{g_{\delta}(b)} b \sqrt{g_{\delta}(b)} v^*$ for brevity.  Note that $||s|| \leq 1$.  
Now for each $x \in \overline{W}$ we have
\begin{eqnarray*}
& & \left| \left| [a_1 s a_1](x) \right| \right| \\ 
& = & \left| \left| [\sqrt{a} + (1-f_1)(\sqrt{b}-\sqrt{a})](x)s(x)[\sqrt{a} + (1-f_1)(\sqrt{b}-\sqrt{a})](x) \right| \right|\\
& = & \left| \left| [\sqrt{a} s \sqrt{a}](x) + r(x) \right| \right|,
\end{eqnarray*}
where $||r(x)|| < 2\sqrt{\epsilon} + \epsilon$.  We revise our definition of $c$ by setting it equal to $a_1$ on $X \backslash U$ and extending it in
an arbitrary fashion to a positive element of $\mathrm{M}_{4n}(\mathrm{C}(X))$.  Combining this new definition of $c$ with (\ref{est1}) above 
we have the estimate
\begin{equation}\label{est2}
\left| \left| [(cvd)b(cvd)^*](x) - a(x) \right| \right| < 2(\sqrt{\epsilon} + \epsilon), \ \forall x \in X \backslash U.
\end{equation}

Choose an open subset $V_2$ of $Z$ such that $\overline{U} \subseteq V_2 \subseteq \overline{V_2} \subseteq V_1$, and a continuous map
$f_2:Z \to [0,1]$ equal to zero on $Z \backslash V_1$ and equal to one on $\overline{V_2}$.  For each $x \in \overline{V_1}$ we have $c(x) =
\sqrt{b(x)}$, $d(x) = \sqrt{g_\delta(b)}$, and $v(x) = \mathbf{1}$, whence
\begin{eqnarray}\label{est3}
\left| \left| [(cvd)b(cvd)^*](x) - a(x) \right| \right| & = & \left| \left| b(x)^2 g_\delta(b)(x) - a(x) \right| \right| \\
& = & \left| \left| b(x)f_\delta(b)(x) - a(x) \right| \right| \\
& \leq & \left| \left| b(x)f_\delta(b)(x) - b(x)\right| \right| + \left| \left| b(x) - a(x) \right| \right| < 2 \epsilon.
\end{eqnarray}
For each $s \in [0,1]$ define
\[
h_s(t) = \left\{ \begin{array}{ll} 2ts/\delta, & t \in [0,\delta/2] \\
(2t-\delta)[(1-s)/(2-\delta)] + s, & t \in (\delta/2,1] \end{array} \right. .
\]
It straightforward to verify that $h_s(t)$ is a homotopy of maps such that 
\[
h_0(t) = t; \ \ h_1(t) =  \left\{ \begin{array}{ll} 2t/\delta, & t \in [0,\delta/2] \\
1, & t \in (\delta/2,1] \end{array} \right. .
\]
Set
\[
g_{\delta,s}(t) = \left\{ \begin{array}{ll} 0, & t \in [0,\delta/2] \\ f_\delta(t)/h_s(t), & t \in (\delta/2,1] \end{array} \right. .
\]
With these definitions we have $h_s(t) g_{\delta,s}(t) = f_\delta(t)$, $\forall s, t \in [0,1]$.  For each $x \in \overline{V_1}$, we adjust our
definitions of $c(x)$ and $d(x)$ as follows:
\[
c(x) = \sqrt{h_{f_2(x)}(b(x))}; \ \  d(x) = \sqrt{g_{\delta,f_2(x)}(b(x))}.
\]
Since $f_2(x) = 0$ on $\partial V_1$, the definitions of $c(x)$ and $d(x)$ are not altered on $\partial V_1$.  Thus, our modified versions of $c$ and
$d$ are still positive elements of $\mathrm{M}_{4n}(\mathrm{C}(Z))$, and the estimate (\ref{est2}) still holds on $Z \backslash V_1$.  For 
$x \in \overline{V_1}$ we have
\begin{eqnarray}\label{est4}
\left| \left| [(cvd)b(cvd)^*](x) - a(x) \right| \right| & = & \left| \left| h_{f_2(x)}(b(x))g_{\delta,f_2(x)}(b(x))b(x) - a(x) \right| \right| \\
& = & \left| \left| f_\delta(b(x))b(x) - a(x) \right| \right| < 2 \epsilon,
\end{eqnarray}
where the last inequality follows from (\ref{est3}) above.  Thus, (\ref{est2}) continues to hold with our new definitions of $c$ and $d$.  

Choose an open subset $V_3$ of $Z$ such that $\overline{U} \subseteq V_3 \subseteq \overline{V_3} \subseteq V_2$, and a continuous map
$f_3:Z \to [0,1]$ equal to zero on $Z \backslash V_2$ and equal to one on $\overline{V_3}$.  For each $s \in [0,1]$ define continuous maps 
$r_s,w_s:[0,1] \to [0,1]$ by 
\[
r_s(t) = \mathrm{max} \left\{ s, \sqrt{h_1(t)} \right\}; \ \ w_s(t) = \mathrm{max} \left\{s, \sqrt{g_{\delta,1}(t)} \right\}.
\]
Thus, $r_s$ and $w_s$ define homotopies of self-maps of $[0,1]$ such that $r_0=h_1$, $w_0=g_{\delta,1}$, and $r_1=w_1= 1$.  For each $x \in \overline{V_2}$
we adjust our definitions of $c(x)$ and $d(x)$ as follows:
\[
c(x) = \sqrt{r_{f_3(x)}(b(x))}; \ \ d(x) = \sqrt{w_{f_3(x)}(b(x))}.
\]
Since $f_3 = 0$ on $\partial V_2$, the definitions of $c(x)$ and $d(x)$ are not altered on $\partial V_2$.  Thus, our modified versions of $c$ and
$d$ are still positive elements of $\mathrm{M}_{4n}(\mathrm{C}(Z))$, and the estimate (\ref{est2}) still holds on $Z \backslash V_2$.  For 
$x \in \overline{V_2}$ we have 
\[
\left| \left| [(cvd)b(cvd)^*](x) - a(x) \right| \right| = \left| \left| r_{f_3(x)}(b(x))w_{f_3(x)}(b(x))b(x) - a(x) \right| \right| < 2 \epsilon
\]
by a functional calculus argument similar to (\ref{est4}) above---one need only observe that
\[
f_\delta(t) \leq r_s(t) w_s(t) \leq 1, \ \forall s,t \in [0,1].
\]
Thus, (\ref{est2}) continues to hold with our new definitions of $c$ and $d$.  Moreover,
we have $c(x) = d(x) = \mathbf{1} \in \mathrm{M}_{4n}(\mathbb{C})$ for each $x \in \overline{V_3}$.  We may thus extend our definitions of $c$ and $d$
to all of $X$ by setting $c(x) = d(x) = \mathbf{1} \in \mathrm{M}_{4n}(\mathbb{C})$ for every $x \in U \cup \overline{V_3}$.  With this final definition
of $c$ and $d$, we see that
\[
\left| \left| [(cvd)b(cvd)^*](x) - a(x) \right| \right| = \left| \left| b(x) - a(x) \right| \right| < \epsilon, \ \forall x \in U \cup \overline{V_3}.
\]
We conclude that the estimate (\ref{est2}) holds on all of $X$, whence {\bf (S)} holds.
\end{proof}

With {\bf (S)} in hand, we have completed the proof of Lemma \ref{main}.

\section{A comparison theorem for recursive subhomogeneous C$^*$-algebras}

\subsection{Background and notation}
Let us recall some of the terminology and results from \cite{P4}.
\begin{dfs}\label{rshdef}
 A {\it recursive subhomogeneous algebra} (RSH algebra) is given by the following recursive definition.
\begin{enumerate} 
    \item[(i)] If $X$ is a compact Hausdorff space and $n \in \mathbb{N}$, then $\mathrm{M}_n(\mathrm{C}(X))$
    is a recursive subhomogeneous algebra.
    \item[(ii)] If $A$ is a recursive subhomogeneous algebra, $X$ is a
    compact Hausdorff space, $X^{(0)} \subseteq X$ is closed, $\phi:A \to \mathrm{M}_k(\mathrm{C}(X^{(0)}))$
    is a unital $*$-homomorphism, and
    $\rho:\mathrm{M}_k(\mathrm{C}(X)) \to \mathrm{M}_k(\mathrm{C}(X^{(0)}))$ is the restriction homomorphism, then
    the pullback
    \[
    A \oplus_{\mathrm{M}_k(\mathrm{C}(X^{(0)}))} \mathrm{M}_k(\mathrm{C}(X))
    =\{(a,f) \in A \oplus \mathrm{M}_k(\mathrm{C}(X)) \ | \
    \phi(a) = \rho(f) \}
    \]
    is a recursive subhomogeneous algebra.
\end{enumerate}
\end{dfs}
\noindent It is clear from the definition above that a C$^*$-algebra $R$ is an RSH algebra if and
only if it can be written in the form
\begin{equation}\label{decomp2}
R = \left[ \cdots \left[  \left[ C_0 \oplus_{C_1^{(0)}} C_1 \right] \oplus_{C_2^{(0)}} C_2 \right]
\cdots \right] \oplus_{C_l^{(0)}} C_l,
\end{equation}
with $C_k = \mathrm{M}_{n(k)}(\mathrm{C}(X_k))$ for compact Hausdorff spaces $X_k$ and integers
$n(k)$, with $C_k^{(0)}=\mathrm{M}_{n(k)}(\mathrm{C}(X_k^{(0)}))$ for compact subsets $X_k^{(0)}
\subseteq X$ (possibly empty), and where the maps $C_k \to C_k^{(0)}$ are always the restriction maps.  
We refer to the expression in (\ref{decomp2}) as a {\it decomposition} for $R$.  Decompositions for RSH
algebras are not unique.

Associated with the decomposition (\ref{decomp2}) are:
\begin{enumerate}
\item[(i)] its {\it length} $l$;
\item[(ii)] its {\it $k^{\mathrm{th}}$ stage algebra}
\[
R_k = \left[ \cdots \left[  \left[ C_0 \oplus_{C_1^{(0)}} C_1 \right] \oplus_{C_2^{(0)}} C_2
\right] \cdots \right] \oplus_{C_k^{(0)}} C_k;
\]
\item[(iii)] its {\it base spaces} $X_0,X_1,\ldots,X_l$ and {\it total space} $\sqcup_{k=0}^l X_k$;
\item[(iv)] its {\it matrix sizes} $n(0),n(1),\ldots,n(l)$ and {\it matrix size function $m:X \to
\mathbb{N}$} given by $m(x) = n(k)$ when $x \in X_k$ (this is called the {\it matrix size of $R$ at $x$});
\item[(v)] its {\it minimum matrix size $\mathrm{min}_k n(k)$} and {\it maximum matrix size $\mathrm{max}_k n(k)$};
\item[(vi)] its {\it topological dimension $\mathrm{dim}(X)$} and {\it topological dimension function $d:X \to \mathbb{N} \cup
\{0\}$} given by $d(x) = \mathrm{dim}(X_k)$ when $x \in X_k$;
\item[(vii)] its {\it standard representation $\sigma_R:R \to \oplus_{k=0}^l \mathrm{M}_{n(k)}(\mathrm{C}(X_k))$} defined to be the obvious inclusion;
\item[(viii)] the {\it evaluation maps $\mathrm{ev}_x:R \to \mathrm{M}_{n(k)}$} for $x \in X_k$, defined to be the composition of evaluation
at $x$ on $\oplus_{k=0}^l \mathrm{M}_{n(k)}(\mathrm{C}(X_k))$ and $\sigma_R$.
\end{enumerate}

\begin{rems}\label{rshmisc} {\rm
  If $R$ is separable, then the $X_k$ can be taken to be
metrisable (\cite[Proposition 2.13]{P4}).  If $R$ has no irreducible representations of
dimension less than or equal to $N$, then we may assume that $n(k)
> N$.  It is clear from the construction of $R_{k+1}$ as a pullback of $R_k$ and $C_{k+1}$ that there is a
canonical surjective $*$-homomorphism $\lambda_k:R_{k+1} \to R_k$.  By composing several such, one has also a canonical surjective 
$*$-homomorphism from $R_j$ to $R_k$ for any $j >k$.  Abusing notation slightly, we denote these maps by $\lambda_k$ as well.}
\end{rems}

\begin{rems}\label{matrixrsh} {\rm
The C$^*$-algebra $\mathrm{M_m}(R) \cong R \otimes \mathrm{M}_m(\mathbb{C})$ is an RSH algebra in a canonical way:  $C_k$ and $C_k^{(0)}$ are
replaced with $C_k \otimes \mathrm{M}_m(\mathbb{C})$ and $C_k^{(0)} \otimes \mathrm{M}_m(\mathbb{C})$, respectively, and the
clutching maps $\phi_k:R_k \to C_{k+1}^{(0)}$ are replaced with the amplifications 
\[
\phi_k \otimes \mathbf{id}_m:C_k \otimes \mathrm{M}_m(\mathbb{C}) \to C_{k+1}^{(0)} \otimes \mathrm{M}_m(\mathbb{C}).
\]
From here on we assume that $\mathrm{M}_m(R)$ is equipped with this canonical decomposition whenever $R$ is given with
a decomposition.  We will abuse notation
by using $\phi_k$ to denote both the original clutching map in the given decomposition for $R$ and its amplified versions.}
\end{rems}

\subsection{A comparison theorem}

\begin{lms}\label{invertextend}
Let $X$ be a compact metrisable Hausdorff space, and $Y$ a closed subset of $X$.  If $a \in \mathrm{M}_{n}(\mathrm{C}(Y))$ is
positive, then $a$ can be extended to $\tilde{a} \in \mathrm{M}_n(\mathrm{C}(X))$
with the property that $\tilde{a}(x)$ is invertible for every $x \in X \backslash Y$.  If $u = v \oplus v^*$ for a
unitary $v \in \mathrm{M}_n(\mathrm{C}(Y))$, then $u$ can be extended to a unitary $\tilde{u} \in \mathrm{M}_{2n}(\mathrm{C}(X))$.
\end{lms}

\begin{proof}
By the semiprojectivity of the C$^*$-algebras they generate, both $a$ and $u$ can be extended to the closure of an open neighbourhood
$V$ of $Y$.  We will also denote these extensions by $a$ and $u$.  Fix a continuous map $f:X \to [0,1]$ which is equal to zero on $Y$,
equal to one on $X \backslash V$, and nonzero at every $x \in X \backslash Y$.  

Define
\[
\tilde{a}(x) = \left\{ \begin{array}{ll} a(x) + f(x)(||a||-a(x)), & x \in \overline{V} \\ ||a||, & x \in X \backslash \overline{V} 
\end{array} \right. 
\]
Clearly, $\tilde{a}$ belongs to $\mathrm{M}_n(\mathrm{C}(X))$ and extends $a$.
It follows that for each $x \in X \backslash Y$, either $\tilde{a}(x) = ||a|| \in \mathrm{GL}_n(\mathbb{C})$, or
\[
\tilde{a}(x) = a(x) + f(x)(||a||-a(x)) = f(x)||a|| + (1-f(x))a(x) \geq f(x)||a|| > 0.
\]
In the latter case we conclude that the rank of $\tilde{a}(x)$ is $n$, whence $\tilde{a}(x) \in \mathrm{GL}_n(\mathbb{C})$
as desired.

Now let us turn to $u$.  We have 
\[
u|_{\overline{V \backslash Y}} = v|_{\overline{V \backslash Y}} \oplus v^*|_{\overline{V \backslash Y}} \sim_h
\mathbf{1} \in \mathrm{M}_{2n}\left( \mathrm{C} \left( \overline{V \backslash Y} \right) \right)
\]
by the Whitehead Lemma, where $\sim_h$ denotes homotopy within the unitary group.  Let $H(x,t): 
\overline{V \backslash Y} \times [0,1] \to \mathcal{U}(\mathrm{M}_{2n}(\mathbb{C}))$ be an implementing 
homotopy, with $H(x,0) = u|_{\overline{V \backslash Y}}$ and $H(x,1) = \mathbf{1}$.  Define
\[
\tilde{u}(x) = \left\{ \begin{array}{ll} u(x), & x \in Y \\ H(x,f(x)), & x \in V \backslash Y \\ \mathbf{1}, & x \in X \backslash V \end{array}
\right.
\]
It is straightforward to check that $\tilde{u}$ is a unitary in $\mathrm{M}_{2n}(\mathrm{C}(X))$, and $\tilde{u}$ extends $u$ by
definition.   
\end{proof}

\begin{lms}\label{mainrsh}
Let $A$ be a separable RSH algebra with a fixed decomposition as above.  Let $a, b \in A$ be positive, and suppose that
$||\lambda_k(b-a)|| < \epsilon$ inside the $k^{\mathrm{th}}$ stage algebra $A_k$, $k < l$.  Suppose further that
\[
\mathrm{rank}(a(x)) + (d(x)-1)/2 \leq \mathrm{rank}(b(x)), \forall x \in X_j \backslash X_j^{(0)}, j > k.
\]
It follows that there are $m \in \mathbb{N}$ and $v \in \mathrm{M}_m(A)$ such that, upon considering $A$ as the
upper-left $1 \times 1$ corner of $\mathrm{M}_m(A)$ we have $||\lambda_{k+1}(vbv^*-a)|| < N\sqrt{\epsilon}$ for the
constant $N$ of Lemma \ref{main} and 
\[
\mathrm{rank}(a(x)) + (d(x)-1)/2 \leq \mathrm{rank}((vbv^*)(x)), \forall x \in X_j \backslash X_j^{(0)}, j > k+1.
\]
\end{lms}

\begin{proof}
Let $\phi_k:A_k \to C_{k+1}^{(0)}$ be the $k^{\mathrm{th}}$ clutching map.  Our hypotheses imply that $\phi_k(b), \phi_k(a) \in
C_{k+1}^{(0)} = \mathrm{M}_{n(k+1)}(\mathrm{C}(X_{k+1}^{(0)})$ satisfy $||\phi_k(b) - \phi_k(a)|| < \epsilon$.  Apply Lemma
\ref{main} with $\phi_k(a), \phi_k(b), X_{k+1}, X_{k+1}^{(0)}$, and $\epsilon$ in place of $a, b, X, Y$, and $\epsilon$, respectively.
The conclusion of Lemma \ref{main} provides us with positive elements $c, d$ and a unitary element $u$ in $\mathrm{M}_{4n(k+1)}(\mathrm{C}(X_{k+1}))$
such that 
\begin{enumerate}
\item[(i)] $||(cud)\phi_k(b)(cud)^* - \phi_k(a)|| < N \sqrt{\epsilon}$, and
\item[(ii)] $c(x) = d(x) = u(x) = \mathbf{1} \in \mathrm{M}_{4n(k+1)}(\mathbb{C})$ for every $x \in X_{k+1}^{(0)}$.
\end{enumerate}
Using (ii) we extend $c,d$, and $u$ to $\mathrm{M}_4(A_{k+1})$ (keeping the same notation) by setting
\[
\lambda_k(c) = \lambda_k(d) = \lambda_k(u) = \mathbf{1} \in \mathrm{M}_4(A_k).
\]
Set $v_{k+1} = cud \in \mathrm{M}_4(A_{k+1})$.  We claim that
\[
||v_{k+1} \lambda_{k+1}(b) v_{k+1}^* - \lambda_{k+1}(a)|| < N \sqrt{\epsilon}.
\] 
It will suffice to prove that the image of $v_{k+1} \lambda_{k+1}(b) v_{k+1}^* - \lambda_{k+1}(a)$
under the standard representation 
\[
\sigma_{\mathrm{M}_4(A_{k+1})}:\mathrm{M}_4(A_{k+1}) \to \bigoplus_{j=0}^{k+1} \mathrm{M}_{4n(j)}(\mathrm{C}(X_j))
\]
is of norm at most $N \sqrt{\epsilon}$.  This in turn need only be checked in each of the direct summands of the 
codomain.  In the summand $\oplus_{j=0}^k \mathrm{M}_{4n(j)}(\mathrm{C}(X_j))$ the desired estimate follows from two
facts:  $\sigma_{\mathrm{M}_4(A_{k+1})}(v_{k+1})$ is equal to the unit of the said summand (see (ii) above), and the images of $a$ and $b$
in this summand are at distance strictly less than $\epsilon < N\sqrt{\epsilon}$.  In the summand $\mathrm{M}_{4n(k+1)}(\mathrm{C}(X_{k+1}))$
the desired estimate follows from (i) above.

If $m \geq 4$, then any $v \in \mathrm{M}_m(A)$ which, upon viewing $\mathrm{M}_4(A)$
as the upper-left $4 \times 4$ corner of $\mathrm{M}_m(A)$, has the property that $\lambda_{k+1}(v)=v_{k+1}$ will
at least satisfy $||\lambda_{k+1}(vbv^*-a)|| < N \sqrt{\epsilon}$.  It remains, then, to find such 
a $v$, while ensuring that 
\[
\mathrm{rank}(a(x)) + (d(x)-1)/2 \leq \mathrm{rank}((vbv^*)(x)), \forall x \in X_j \backslash X_j^{(0)}, j > k+1.
\]

If $k+1=l$, then there is nothing to prove.  Suppose that $k+1 < l$.  
Let us first construct an element $v_{k+2}$ of $\mathrm{M}_{8}(A_{k+2})$ with the following properties:
$\lambda_{k+1}(v_{k+2}) = v_{k+1}$, and 
\[
\mathrm{rank}(a(x)) + (d(x)-1)/2 \leq \mathrm{rank}((v_{k+2}bv_{k+2}^*)(x)), \forall x \in X_{k+2} \backslash X_{k+2}^{(0)}.
\]

Define $c_{k+1} = c \oplus 0$, $d_{k+1} = d \oplus 0$, and $u_{k+1} = u \oplus u^*$.  Use Lemma \ref{invertextend} to extend
$\phi_{k+1}(c_{k+1})$, $\phi_{k+1}(d_{k+1})$, and $\phi_{k+1}(u_{k+1})$ to positive elements $\tilde{c}_{k+2}, \tilde{d}_{k+2}$
and a unitary element $\tilde{u}_{k+2}$, respectively, in $\mathrm{M}_{8n(k+2)}(\mathrm{C}(X_{k+2}))$, all of which are invertible
at every $x \in X_{k+2} \backslash X_{k+2}^{(0)}$.  Consider $\mathrm{M}_8(A_{k+2})$ as a subalgebra of $\oplus_{j=0}^{k+2}
\mathrm{M}_{8n(j)}(\mathrm{C}(X_{j}))$ via its standard representation, and define $c_{k+2}$ to be equal to $c_{k+1}$ in the first
$k+1$ summands, and equal to $\tilde{c}_{k+2}$ in the last summand;  define $d_{k+2}$ and $u_{k+2}$ similarly.  Setting $v_{k+2} =
c_{k+2} u_{k+2} d_{k+2}$ we have that 
\begin{eqnarray*}
\lambda_{k+1}(v_{k+2}) & = & \lambda_{k+1}(c_{k+2} u_{k+2} d_{k+2}) \\
& =  & c_{k+1} u_{k+1} d_{k+1} \\
& = & (c \oplus 0)(u \oplus u^*)(d \oplus 0) \\
& = & cud \oplus 0 \\
& = & v_{k+1}.
\end{eqnarray*}
Moreover, for each $x \in X_{k+2} \backslash X_{k+2}^{(0)}$, we have
\[
v_{k+2}(x) = \tilde{c}_{k+2}(x)\tilde{u}_{k+2}(x)\tilde{d}_{k+2}(x) \in \mathrm{GL}_{8n(k+2)}(\mathbb{C}).
\]
It follows that 
\[
\mathrm{rank}((vbv^*)(x)) = \mathrm{rank}(b(x)) \geq (d(x)-1)/2 + \mathrm{rank}(a(x)),  \ \forall x \in X_{k+2} \backslash X_{k+2}^{(0)},
\]
as required.

If $k+2 = l$ then we set $v = v_{k+2}$ to complete the proof.  Otherwise, we repeat the arguments in the  paragraph above using $c_{k+2}, 
d_{k+2},$ and $u_{k+2}$ in place of $c, d,$ and $u$, respectively, to obtain $v_{k+3} \in \mathrm{M}_{8^2}(A_{k+3})$ such that 
$\lambda_{k+2}(v_{k+3})=v_{k+2}$ and
\[
\mathrm{rank}(a(x)) + (d(x)-1)/2 \leq \mathrm{rank}((v_{k+3}bv_{k+3}^*)(x)), \forall x \in X_{k+3} \backslash X_{k+3}^{(0)}.
\]
Continuing this process until we arrive at $v_{k + (l-k)} = v_l$ and setting $v = v_l$ yields the Lemma in full. 
\end{proof}

\begin{thms}\label{rshcomp}
Let $A$ be a separable RSH algebra with a fixed decomposition as above.  Let $a, b \in A$ be positive, and suppose that 
\[
\mathrm{rank}(a(x)) + (d(x)-1)/2 \leq \mathrm{rank}(b(x)), \ \forall x \in X_k \backslash X_k^{(0)}, \ k \in \{0,1,\ldots,l\}.
\]
It follows that $a \precsim b$.
\end{thms}

\begin{proof}  We view $A$ as the upper-left $1 \times 1$ corner of $\mathrm{M}_m(A)$, and adopt the standard notation for the
decompositions of $A$ and $\mathrm{M}_m(A)$.  Let $\epsilon > 0$ be given;  we must find $m \in \mathbb{N}$ and $v \in \mathrm{M}_m(A)$ 
such that $||vbv^*-a||< \epsilon$.

Let $l$ be the length of the fixed decomposition for $A$.  Given $\delta_0 > 0$, we define $\delta_k = N\sqrt{\delta_{k-1}}$
for each $k \in \{1,\ldots,l\}$, where $N$ is the constant of Lemma \ref{main}.  It follows that 
\[
\delta_k = \delta_0^{1/2^k} \prod_{j=0}^{k-1} N^{1/2^j}.
\] 
Assume that $\delta_0$ has been chosen so that $\delta_l < \epsilon$. 

Apply Lemma \ref{main} with $\lambda_0(a), \lambda_0(b), X_0,$ and $\emptyset,$ in place of $a, b, X,$ and $Y$.  Since $Y$ is 
empty, we can arrange to have any value of $\epsilon$ appear in the conclusion of Lemma \ref{main}.  We choose 
$\epsilon = \delta_0^2/N^2$, so that the norm estimate in the conclusion of Lemma \ref{main} is strictly less than $N \sqrt{\delta_0^2/N^2}
= \delta_0$.  Let $c_0, d_0,$ and $u_0$ denote the positive elements and the unitary element, respectively, of 
$\mathrm{M}_{4n(0)}(\mathrm{C}(X_0))$ produced by Lemma \ref{main}.  Apply the arguments of the second-to-last paragraph in the proof
of Lemma \ref{mainrsh} with $c_0, d_0,$ and $u_0$ in place of $c, d,$ and $u$, respectively, to produce an element $v_0$ of
$\mathrm{M}_{32}(A)$ such that $||\lambda_0(v_0 b v_0^* - a)|| < \delta_0$, and 
\[
\mathrm{rank}(a(x)) + (d(x)-1)/2 \leq \mathrm{rank}((v_0bv_0^*)(x)), \forall x \in X_j \backslash X_j^{(0)}, j > 0.
\]

Suppose that we have found $m_k \in \mathbb{N}$ and $v_k \in \mathrm{M}_{m_k}(A)$ such that
$||\lambda_k(v_k b v_k^* - a)|| < \delta_k$ and 
\[
\mathrm{rank}(a(x)) + (d(x)-1)/2 \leq \mathrm{rank}((v_kbv_k^*)(x)), \forall x \in X_j \backslash X_j^{(0)}, j > k.
\]
An application of Lemma \ref{mainrsh} yields $v_{k+1} \in \mathrm{M}_{8m_k}(A)$ such that 
$||\lambda_{k+1}(v_{k+1}v_k b v_k^* v_{k+1}^* -a)|| < N \sqrt{\delta_k} = \delta_{k+1}$ and
\[
\mathrm{rank}(a(x)) + (d(x)-1)/2 \leq \mathrm{rank}((v_{k+1}v_k b v_k^* v_{k+1}^*)(x)), \ \forall x \in X_j \backslash X_j^{(0)}, j > k+1.
\]
Starting with  $v_0$, we use the fact above to find, successively, $v_1,\ldots,v_l$.  With $v = v_l v_{l-1} \cdots v_0$ we have
\[
||vbv^*-a|| < \delta_l < \epsilon,
\]
as desired.
\end{proof} 

\section{Applications}\label{apps}

\subsection{The radius of comparison and strict comparison}\label{rcsec}

Let $A$ be a unital stably finite C$^*$-algebra, and let $a,b \in \mathrm{M}_{\infty}(A)$ be positive.  
We say that $A$ has {\it $r$-strict comparison} if $a \precsim b$ whenever
\[
d(a) + r < d(b), \ \forall d \in \mathrm{LDF}(A).
\]
The {\it radius of comparison of $A$}, denoted by $\mathrm{rc}(A)$, is defined to be the
infimum of the set 
\[
\{r \in \mathbb{R}^+ \ | \ A \ \mathrm{has} \ \mathrm{r-strict \ comparison} \}
\]
whenever this set is nonempty;  if the set is empty then we set $\mathrm{rc}(A) = \infty$ (\cite{To3}).
The condition $\mathrm{rc}(A) = 0$ is equivalent to $A$ having strict comparison (see Subsection \ref{dimfunc}).

The radius of comparison should be thought of as the ratio of the topological dimension of $A$ to its
matricial size, despite the fact that both may be infinite.  It has been useful in distinguishing
C$^*$-algebras which are not $\mathrm{K}$-theoretically rigid in the sense of G. A. Elliott (\cite{kg}, 
\cite{To4}).    Here we give sharp upper bounds
on the radius of comparison of a recursive subhomogeneous algebra.  These improve significantly upon the upper bounds 
established in the homogeneous case by \cite[Theorem 3.15]{To5}.

\begin{thms}\label{rshrc}
Let $A$ be a separable RSH algebra with a fixed decomposition of length $l$ and matrix sizes $n(0),\ldots,n(l)$.  
It follows that 
\[
\mathrm{rc}(A) \leq \underset{0\leq k\leq l}{\max} \frac{\mathrm{dim}(X_k)-1}{2n(k)}.
\]
\end{thms}

\begin{proof}
Use $r$ to denote the upper bound in the statement of the theorem, and suppose that we are given $a,b \in \mathrm{M}_{\infty}(A)_+$
such that $d_\tau(a) + r < d_\tau(b)$ for every $\tau \in \mathrm{T}(A)$.  Associated to each $x \in X_k \backslash X_k^{(0)}$, $0 
\leq k \leq l$, is an extreme point of $\mathrm{T}(A)$, denoted by $\tau_x$, obtained by composing $ev_x$ with the normalised
trace on $\mathrm{M}_{n(k)}$.  For any $a \in \mathrm{M}_{\infty}(A)_+$ we have $d_{\tau_x}(a) = [\mathrm{rank}(ev_x(a))]/n(k)$,
and so 
\[
\frac{\mathrm{rank}(ev_x(a))}{n(k)} + \frac{\mathrm{dim}(X_k)-1}{2n(k)} \leq  \frac{\mathrm{rank}(ev_x(a))}{n(k)} +r <
\frac{\mathrm{rank}(ev_x(b))}{n(k)}.
\]
Multiplying through by $n(k)$ we have
\[
\mathrm{rank}(a(x)) + \frac{\mathrm{dim}(X_k)-1}{2} < \mathrm{rank}(b(x))
\]
for every $x \in X_k \backslash X_k^{(0)}$ and $k \in \{0,\ldots,l\}$, whence $a \precsim b$ by Theorem \ref{rshcomp},
as desired.
\end{proof}

\noindent
Specialising to the homogeneous case we have the following corollary.

\begin{cors}\label{homsharp}
Let $X$ be a compact metrisable Hausdorff space of covering dimension $d \in \mathbb{N}$, and $p \in \mathrm{C}(X) \otimes \mathcal{K}$
a projection.  If follows that
\[
\mathrm{rc}(p(\mathrm{C}(X) \otimes \mathcal{K})p) \leq \frac{d-1}{2 \mathrm{rank}(p)}.
\]
\end{cors}

\begin{proof}
The algebra $p(\mathrm{C}(X) \otimes \mathcal{K})p$ admits a recursive subhomogeneous decomposition in
which every matrix size is equal to $\mathrm{rank}(p)$ and each $X_k$ has covering dimension at most $d$.
(This decomposition comes from the fact that $p$ corresponds to a vector bundle of finite type---see 
Section 2 of \cite{P4}.)  The Corollary now follows from Theorem \ref{rshrc}.
\end{proof}

\noindent
Corollary \ref{homsharp} improves upon \cite[Theorem 3.15]{To5}, or rather, the upper bound on the radius
of comparison that can be derived from it:  the latter result leads to an upper bound of $(9d)/\mathrm{rank}(p)$.

The property of strict comparison is a powerful regularity property with agreeable consequences.  We
will see some examples of this in Subsections \ref{cuntz}, \ref{bh}, and \ref{hilbert};  a fuller treatment of this
topic can be found in \cite{ET}.  

\begin{thms}\label{ashcomp}
Let $(A_i,\phi_i)$ be a unital direct sequence of recursive subhomogeneous algebras with slow dimension 
growth.  If $A = \lim_{i \to \infty}(A_i, \phi_i)$ is simple, then $A$ has strict comparison of positive elements.
\end{thms}

\begin{proof}
Let us first show that $\liminf_{i \to \infty} \ \mathrm{rc}(A_i) = 0$.  We assume that each $A_i$ is equipped with a fixed decomposition.
Let $Y_i = \sqcup_{k=0}^{l_i} X_{i,k}$ denote the total space of $A_i$, $d_i:Y_i \to \{0\} \cup \mathbb{N}$ its 
topological dimension function, and $n_i(0),\ldots,n_i(l_i)$ its matrix sizes.  
From \cite[Definition 1.1]{P5}, $(A_i,\phi_i)$ has slow dimension
growth if the following statement holds:  for every $i \in \mathbb{N}$, projection $p \in \mathrm{M}_{\infty}(A_i)$,
and $N \in \mathbb{N}$, there exists $j_0 > i$ such that for every $j \geq j_0$ and $y \in Y_i$ we have 
\[
ev_y(\phi_{i,j}(p)) = 0 \ \ \mathrm{or} \ \ \mathrm{rank}(ev_y(\phi_{i,j}(p))) \geq N d_j(y);
\]
if $p = \mathbf{1}_{A_i}$, then only the latter statement can hold.  If $y \in X_{j,k} \backslash X_{j,k}^{(0)}$,
then 
\[
\mathrm{rank}(ev_y(\phi_{i,j}(\mathbf{1}_{A_i}))) = \mathrm{rank}(ev_y(\mathbf{1}_{A_j})) = n_j(k) \geq N \mathrm{dim}(X_{j,k}).
\]
It now follows from Theorem \ref{rshrc} that $\liminf_{i \to \infty} \ \mathrm{rc}(A_i) = 0$.

Theorem 4.5 of \cite{To5} would give us strict comparison for $A$ if only each $\phi_i$ were injective.  
The origin of this injectivity hypothesis lies in \cite[Lemma 4.4]{To5}---the proof of \cite[Theorem 4.5]{To5} only 
uses injectivity of the $\phi_i$ in its appeal to this Lemma.  Thus, we must drop injectivity from the assumptions 
of \cite[Lemma 4.4]{To5};  we must prove the following claim:

\vspace{2mm}
\noindent
{\bf Claim:}  Let $B$ be the limit of an inductive sequence $(B_i,\psi_i)$ of C$^*$-algebras, and let $a,b \in \mathrm{M}_{\infty}(B)$
be positive.  If $\psi_{i,\infty}(a) \precsim \psi_{i,\infty}(b)$, then for every $\epsilon > 0$ there is a $j>i$ such that
$(\psi_{i,j}(a)-\epsilon)_+ \precsim \psi_{i,j}(b)$.

\vspace{2mm}
\noindent
{\it Proof of claim.}  If will suffice to prove the claim for $a,b \in B$.  By assumption, there is a sequence $(v_k)$ in 
$B$ such that $v_k b v_k^* \to a$.  We may assume that the $v_k$ lie in the dense local C$^*$-algebra $\cup_i \ \psi_{i,\infty}(B_i)$
(see the proof of \cite[Lemma 4.4]{To5}).  In fact, by compressing our inductive sequence, we may as well assume that
$v_k = \phi_{k,\infty}(w_k)$ for some $w_k \in B_k$.  The statement that $v_k b v_k^* \to a$ can now amounts to
\[
\Vert \psi_{k,\infty}( w_k \psi_{i,k}(b) w_k^* - \psi_{i,k}(a)) \Vert \stackrel{n \to \infty}{\longrightarrow} 0.
\]
Fix $k_0$ large enough that the left hand side above is $< \epsilon$.  Since $\Vert \psi_{k_0,j}(x) \Vert \to \Vert \psi_{k_0,\infty}(x) \Vert$
for any $x \in A_{k_0}$ we may find $j >i$ such that
\[ 
\Vert \psi_{k_0,j}( w_{k_0} \psi_{i,k_0}(b) w_{k_0}^* - \psi_{i,k_0}(a)) \Vert < \epsilon.
\]
Setting $r_j = \psi_{k_0,j}(w_{k_0})$ and appealing to part (iii) of Proposition \ref{basics} we have 
\[
(\psi_{i,j}(a)-\epsilon)_+ \precsim r_j \psi_{i,j}(b) r_j^* \precsim \psi_{i,j}(b), 
\]
as desired.  This
proves the claim, and hence the theorem.
\end{proof}

We collect an improvement of \cite[Theorem 4.5]{To5} as a corollary.

\begin{cors}
Let $A$ be the limit of an inductive sequence of stably finite C$^*$-algebras $(A_i,\phi_i)$, with
each $A_i$ and $\phi_i$ unital.  Suppose that $A$ is simple, and that 
\[
\underset{i \to \infty}{\liminf} \  \mathrm{rc}(A_i) = 0.
\]
It follows that $A$ has strict comparison of positive elements. 
\end{cors}

\begin{proof}
Follow the proof of \cite[Theorem 4.5]{To5} but use the claim in the proof of Theorem \ref{ashcomp}
instead of \cite[Lemma 4.4]{To5}.
\end{proof}

\begin{cors}
Let $M$ be a compact smooth connected manifold and $h:M \to M$ a minimal diffeomorphism.  It follows that
the transformation group C$^*$-algebra C$^*(M, \mathbb{Z}, h)$ has strict comparison of positive elements.
\end{cors}

\begin{proof}
By the main result of \cite{LP}, C$^*(M,\mathbb{Z},h)$ can be written as the limit of an inductive sequence of recursive
subhomogeneous C$^*$-algebras with slow dimension growth.  Apply Theorem \ref{ashcomp}.
\end{proof}

\subsection{The structure of the Cuntz semigroup}\label{cuntz}

The Cuntz semigroup is a sensitive invariant in the matter of distinguishing simple separable
amenable C$^*$-algebras, and has recently received considerable attention (see \cite{BPT}, \cite{BT}, \cite{CE}, 
\cite{CEI}, \cite{CRS}, \cite{ET}, \cite{To2}, and \cite{To5}).  It is, however, very difficult to compute in general---see
\cite[Lemma 5.1]{To2}.  This situation improves dramatically in the case of simple C$^*$-algebras with strict comparison
of positive elements.

Let $A$ be a unital, simple, exact, stably finite C$^*$-algebra.  In this case we may write $W(A) = V(A) \sqcup W(A)_+$ (as sets),
where $V(A)$ denotes the semigroup of Murray-von Neumann equivalence classes of projections in $\mathrm{M}_{\infty}(A)$---here
interpreted as the those Cuntz equivalence classes represented by a projection---and 
$W(A)_+$ denotes the subsemigroup of $W(A)$ consisting of Cuntz classes represented by positive elements having zero as an
accumulation point of their spectrum (cf. \cite{PT}).   Let $\mathrm{LAff}_b(\mathrm{T}(A))_{++}$ denote the set of
lower semicontinuous, affine, bounded, strictly positive functions on the tracial state space of $A$, and define a map
$\iota: W(A) \to \mathrm{LAff}_b(\mathrm{T}(A))_{++}$ by $\iota(\langle a \rangle)(\tau) = d_\tau(a)$.  We endow the set
\[
V(A) \sqcup \mathrm{LAff}_b(\mathrm{T}(A))_{++}
\]
with an Abelian binary operation $+_W$ which restricts to the usual semigroup operation in each component and is given by 
$x +_W f = \iota(x) + f $ for $x \in V(A)$ and $f \in \mathrm{LAff}_b(\mathrm{T}(A))_{++}$.  We also define a partial order $\leq_W$ on this set
which restricts to the usual partial orders in each component and satisfies
\begin{enumerate}
\item[(i)] $x \leq_W f$ if and only if $\iota(x) < f$, and
\item[(ii)] $x \geq_W f$ if and only if $\iota(x) \geq f$.
\end{enumerate}

\begin{thms}[Brown-Perera-T \cite{BPT}, Coward-Elliott-Ivanescu \cite{CEI}]\label{cuntzembed}
Let $A$ be a simple, unital, exact, and stably finite C$^*$-algebra with strict comparison of positive elements.  It follows that the map
\[
V(A) \sqcup W(A)_+ \stackrel{\mathrm{id} \sqcup \iota}{\longrightarrow} V(A) \sqcup \mathrm{LAff}_b(\mathrm{T}(A))_{++}
\]
is a semigroup order embedding.
\end{thms}

If $A$ is infinite-dimensional and monotracial, then the embedding of Theorem \ref{cuntzembed} is an isomorphism.  We suspect that the
monotracial assumption is unneccessary.  Theorem \ref{cuntzembed} applies to ASH algebras as in Theorem \ref{ashcomp}, and so to the
minimal diffeomorphism C$^*$-algebras C$^*(M,\mathbb{Z},h)$ considered above.

\subsection{A conjecture of Blackadar-Handelman}\label{bh}

Blackadar and Handelman conjectured in 1982 that the lower semicontinuous dimension functions on a 
C$^*$-algebra should be dense in the set of all dimension functions.  This conjecture was proved 
for C$^*$-algebras as in Theorem \ref{cuntzembed} in \cite[Theorem 6.4]{BPT}.  Thus, we have the 
following result.

\begin{thms} Let $A$ be a C$^*$-algebra as in Theorem \ref{cuntzembed} (in particular, $A$ could 
be the C$^*$-algebra of a minimal diffeomorphism).  It follows that the lower semicontinuous dimension 
functions on $A$ are weakly dense in the set of all dimension functions on $A$.
\end{thms}

\subsection{Classifying Hilbert modules}\label{hilbert}

In \cite{CEI}, Coward, Elliott, and Ivanescu gave a new presentation of the Cuntz semigroup.  Given a C$^*$-algebra $A$, they considered positive
elements in $A \otimes \mathcal{K}$ (as opposed to $\mathrm{M}_{\infty}(A)$, as we have done---the difference is ultimately immaterial).  
If $A$ is separable, then the hereditary subalgebras of
$A \otimes \mathcal{K}$ are singly generated, and any two generators of a fixed hereditary subalgebra are Cuntz equivalent.  
Thus, Cuntz equivalence factors through the passage from a positive element to the hereditary subalgebra it generates.  
These hereditary subalgebras are in one-to-one 
correspondence with countably generated Hilbert $A$-modules, and in \cite{CEI} the notion of Cuntz equivalence, considered as a relation on
hereditary subalgebras, is translated into a relation on Hilbert modules.  Thus, we may speak of Cuntz equivalence between countably generated
Hilbert $A$-modules.

\begin{thms}[Coward-Elliott-Ivanescu, \cite{CEI}]\label{stabrankone}
Let $A$ be a C$^*$-algebra of stable rank one.  It follows that countably generated Hilbert $A$-modules $X$ and $Y$ are Cuntz equivalent if and 
only if they are isomorphic.
\end{thms}
 
 \begin{cors}\label{stabrankonecor}
 Let $A$ be as in Theorem \ref{ashcomp}.  Suppose further that $A$ has stable rank one.  
(In particular, $A$ could by the C$^*$-algebra of a 
minimal diffeomorphism, as these have stable rank one by the main result of \cite{P5}.)  
It follows that countably generated Hilbert $A$-modules 
$X$ and $Y$ are isomorphic if and only if they are Cuntz equivalent.
 \end{cors}
 
 If $X$ and $Y$ as in Corollary \ref{stabrankonecor} are finitely generated and projective, 
then they are Cuntz equivalent if and only if the projections
 in $A \otimes \mathcal{K}$ which generate them as closed right ideals have the same $\mathrm{K}_0$-class.  
Otherwise, $X$ has associated
 to it an affine function on the tracial state space of $A$:  one extends the map $\iota$ of Subsection \ref{cuntz} to have domain
 $A \otimes \mathcal{K}$, applies it to any positive element of $A \otimes \mathcal{K}$ which generates $X$ as a closed right ideal.
 This function determines non-finitely generated $X$ up to isomorphism.  This classification
 of Hilbert $A$-modules is analogous to the  classification of W$^*$-modules over a $\mathrm{II}_1$ factor.  
We refer the reader to Section 3 of \cite{BT} for further details.

\subsection{Classifying self-adjoints}\label{selfadjoint}
 
 We say that self-adjoint elements $a$ and $b$ in a unital C$^*$-algebra $A$ are approximately unitarily equivalent if there is a sequence 
 $(u_n)_{n=1}^{\infty}$ of unitaries in $A$ such that $u_nau_n^* \to b$.  For $a \in A_+$ we let $\phi_a: \mathrm{C}^*(a,\mathbf{1}) \hookrightarrow A$
denote the canonical embedding.  Denote by $\mathbf{Ell}(a)$ the following pair of induced maps:
\[
\mathrm{K}_0(\phi_a): \mathrm{K}_0(\mathrm{C}^*(a,\mathbf{1})) \to \mathrm{K}_0(A); \ \ \phi_a^{\sharp}: \mathrm{T}(A) \to 
\mathrm{T}(\mathrm{C}^*(a,1)).
\]

\begin{thms}[Brown-T, \cite{BT}]\label{auclass}
Let $A$ be a unital simple exact C$^*$-algebra of stable rank one and strict comparison (in particular, $A$ could have stable rank one and
satisfy the hypotheses of Theorem \ref{ashcomp}).  If $a,b \in A_+$, then $a$ and $b$ are approximately unitarily equivalent if and only if
$\sigma(a) = \sigma(b)$ and $\mathbf{Ell}(a) = \mathbf{Ell}(b)$.  
\end{thms}
 
\subsection{The range of the radius of comparison, with applications}\label{anticlasssec}

The classification theory of operator algebras is a rich field.  It was begun by Murray and von Neumann with their type classification
of factors in the 1930s, and has been active ever since.  In the presence of certain regularising assumptions, the theory is well-behaved.  For
instance, there is a complete classification of injective factors with separable predual (due to Connes and Haagerup---see \cite{co2} and
\cite{ha2}), and a similarly successful classification program for simple C$^*$-algebras upon replacing injectivity and separability of the 
predual with amenability and norm-separability, respectively (see \cite{ET} and \cite{R1}).  

Without these regularising assumptions, the theory is fractious, but nonetheless interesting.  One of the landmarks on this side of
the theory is McDuff's construction of uncountably many non-isomorphic factors of type $\mathrm{II}_1$ (\cite{Mc}). 
(More recently there is Popa's work on $\mathrm{II}_1$ factors with Betti numbers invariants---see \cite{P}.) One might view McDuff's result as
saying that there are uncountably many non-isomorphic factors which all have the same naive invariant, namely, the mere fact that they
are $\mathrm{II}_1$ factors.  (Connes proved that there is only one injective $\mathrm{II}_1$ factor with separable predual.)
Here we prove an analogue of McDuff's theorem for simple, separable, amenable C$^*$-algebras, where the corresponding naive invariant
consists of Banach algebra $\mathrm{K}$-theory and positive traces.  We even obtain a somewhat stronger result, replacing non-isomorphism
with non-Morita-equivalence.  In passing we prove that the range of the radius of comparison is exhausted by simple C$^*$-algebras, a result which
represents the first exact calculations of the radius of comparison for any simple C$^*$-algebra.

Recall that the Elliott invariant of a C$^*$-algebra $A$ is the 4-tuple
\begin{equation}\label{ell}
\mathrm{Ell}(A) := \left( (\mathrm{K}_0A,\mathrm{K}_0A^+,\Sigma_A),\mathrm{K}_1A,\mathrm{T}^+A,\rho_A \right),
\end{equation}
where the $\mathrm{K}$-groups are the Banach algebra ones, $\mathrm{K}_0A^+$ is the image
of the Murray-von Neumann semigroup $\mathrm{V}(A)$ under the Grothendieck map, $\Sigma_A$ is
the subset of $\mathrm{K}_0A$ corresponding to projections in $A$, $\mathrm{T}^+A$ is
the space of positive tracial linear functionals on $A$, and $\rho_A$ is the natural pairing
of $\mathrm{T}^+A$ and $\mathrm{K}_0A$ given by evaluating a trace at a $\mathrm{K}_0$-class.

\begin{thms}\label{anticlass}
There is a family $\{A^{(r)}\}_{r \in \mathbb{R}^+ \backslash \{0\}}$ of simple, separable, amenable C$^*$-algebras such that 
$\mathrm{rc}(A_r) = r$ and $\mathrm{Ell}(A_r) \cong \mathrm{Ell}(A_s)$ for every $s, r \in \mathbb{R}^+ \backslash \{0\}$.  
In particular, $A_r \ncong A_s$ whenever $r \neq s$.  If $A_s$ and $A_r$ are Morita equivalent, then $s/r \in \mathbb{Q}$.
\end{thms}

\begin{proof}
The general framework for the construction of $A^{(r)}$ follows \cite{V1}. 
Find sequences of natural numbers $(n_i)$ and $(l_i)$ and a natural number $m_0$ with the following properties:
\begin{enumerate}
\item[(i)] $n_i \to \infty$;
\item[(ii)]
\[
\frac{n_0}{2m_0} \cdot \frac{n_1 n_2 \cdots n_i}{(n_1+l_1)(n_2+l_2) \cdots (n_i+l_i)} \stackrel{i \to \infty}{\longrightarrow} r;
\]
\item[(iii)] $l_i \neq 0$ for infinitely many $i$;
\item[(iv)] every natural number divides some $m_i := m_0(n_1+l_1)(n_2+l_2) \cdots (n_i+l_i)$
\end{enumerate}
Set $X_1 = [0,1]^{n_0}$ and set $X_{i+1} = (X_i)^{n_{i+1}}$.  
Let $\pi_i^j:X_{i+1} \to X_i$, $1 \leq j \leq n_{i+1}$ be the co-ordinate
projections.  Let $A_i$ be the homogeneous C$^*$-algebra $\mathrm{M}_{m_i}(\mathrm{C}(X_i))$, and let $\phi_i:A_i \to A_{i+1}$
be the $*$-homomorphism given by 
\[
\phi_i(f)(x) = \mathrm{diag}\left( f \circ \pi_i^1(x),\ldots,f \circ \pi_i^{n_{i+1}}(x),a(x_i^1),\ldots,a(x_i^{l_i}) \right), \ \forall x \in X_{i+1},
\]
where $x_i^1,\ldots,x_i^{l_i} \in X_i$ are to be specified.  Set $A^{(r)} = \lim_{i \to \infty} (A_i,\phi_i)$,
and define
\[
\phi_{i,j} := \phi_{j-1} \circ \cdots \circ \phi_i.
\]
Let $\phi_{i,\infty}:A_i \to A$ be the canonical map.
We note that the $x_i^1,\ldots,x_i^{l_i} \in X_i$ may be chosen to ensure that $A$ is simple (cf. \cite{V1}); we assume that
they have been so chosen, whence $A^{(r)}$ is unital, simple, separable, and amenable.

By Theorem \ref{rshrc}, we have 
\[
\underset{i \to \infty}{\lim} \mathrm{rc}(A_i) = \underset{i \to \infty}{\lim} \frac{n_0 n_1 \cdots n_i -1}{2m_0 (n_1+l_1)(n_2+l_2) \cdots (n_i+l_i)} = r.
\]
Since the construction of $A^{(r)}$ is the same as that of \cite[Theorem 4.1]{To4}, we conclude that $\mathrm{rc}(A^{(r)}) \leq r$ by 
\cite[Proposition 3.3]{To4}.  

Let $\eta > 0$ be given.  We will exhibit positive elements $a,b \in \mathrm{M}_{\infty}(A^{(r)})$ with the property that
\[
d_\tau(a) + r - \eta < d_\tau(b), \ \forall \tau \in \mathrm{T}(A^{(r)}),
\]
and yet $\langle a \rangle \nleq \langle b \rangle$ in $W(A^{(r)})$.  This will show that $\mathrm{rc}(A^{(r)}) \geq r-\eta$ for 
every $\eta >0$, whence $\mathrm{rc}(A^{(r)}) =r$, as desired.

Choose $i$ large enough that 
\[
\frac{ \lfloor \mathrm{dim}(X_{i})/2 \rfloor -1}{m_i} > r- \eta/4.
\]
It follows from \cite[Theorem 6.6]{To3} that there are $a,b \in \mathrm{M}_{\infty}(A_i)_+$ such that $\langle a \rangle \nleq \langle b
\rangle$ in $W(A_i)$ and yet
\[
d_{\tau}(a) + r-\eta < d_{\tau}(b), \ \forall \tau \in \mathrm{T}(A_i).
\]
Assumption (ii) above ensures that $n_i \neq 0$, whence each $\phi_i$ is injective.  We may thus identify $a$ and $b$ with their
images in $A^{(r)}$ so that
\[
d_{\tau}(a) + r-\eta < d_{\tau}(b), \ \forall \tau \in \mathrm{T}(A^{(r)}).
\]
We need only prove that $\langle a \rangle \nleq \langle b \rangle$ in $W(A^{(r)})$.
The technique for proving this is an adaptation of Villadsen's Chern class obstruction argument
from \cite{V1}.  

With $N_i:= n_0 n_1 \cdots n_i$, we have $A_i = \mathrm{M}_{m_i}(\mathrm{C}([0,1]^{N_i}))$.  The element $b$ of $\mathrm{M}_{\infty}(A_i)$
has the following properties:  there is a closed subset $Y$ of $[0,1]^{N_i}$ homeomorphic to $\mathrm{S}^{2k}$, $N_i-2 \leq 2k \leq N_i$, such that
the restriction of $b$ to $Y$ is a projection of rank $k$ corresponding to the $k$-dimensional Bott bundle $\xi$ over $\mathrm{S}^{2k}$; and
the rank of $b$ is at most $2k$ over any point in $X_i = [0,1]^{N_i}$.  The element $a$ has constant rank---it is a projection
corresponding to a trivial line bundle over $X_i$---and need only have normalised rank strictly less than $3\eta/4$.  By increasing $i$, and hence $m_i$,
if necessary, we may assume that the normalised rank of $a$ is at least $\eta/2$.  This leads to
\begin{equation}\label{rankbound}
d_{\tau}(a) > \frac{\eta}{2} = 2r \left( \frac{\eta}{4r} \right) \geq d_{\tau}(b) \left( \frac{\eta}{4r} \right), \ \forall \tau \in \mathrm{T}(A_i).
\end{equation}

The map $\phi_{i,j}:A_i \to A_j$ has the form
\[
\phi_{i,j}(f) = \mathrm{diag}\left( f \circ \pi_{i,j}^1(x),\ldots,f \circ \pi_{i,j}^{k_{i,j}}(x),f(x_i^1),\ldots,f(x_i^{l_{i,j}}) \right), \ 
\forall x \in X_{j},
\]
where $k_{i,j} = n_{i+1}n_{i+2}\cdots n_j$ and $l_{i,j} = m_j/m_i - k_{i,j}$.  Following \cite{V1}, we have that $\phi_{i,j}(a)$ is a projection
of rank $\mathrm{rank}(a)m_j/m_i$ corresponding to the trivial vector bundle $\theta_{\mathrm{rank}(a) m_j/m_i}$, while the restriction of 
$\phi_{i,j}(b)$ to $Y^{k_{i,j}} \subseteq X_i^{k_{i,j}} = X_j$ is of the form $\xi^{\times k_{i,j}} \oplus f_b$, where $f_b$ is a 
constant positive element of 
rank at most $2k l_{i,j}$.  If $p$ is the image of $\mathbf{1} \in A_i$ under the eigenvalue maps of $\phi_{i,j}$ which are co-ordinate projections, 
then $p\phi_{i,j}(b)p = \xi^{\times k_{i,j}}$.  Let $x \in A_j$.
Restricting to $Y^{k_{i,j}}$ (and using the same notation for the restriction of $x$) we have
\begin{eqnarray*}
& & \Vert x (\xi^{\times k_{i,j}} \oplus f_b)x^* - \theta_{\mathrm{rank}(a) m_j/m_i} \Vert \\
 & = & \Vert [x(p \oplus f_b^{1/2})](\xi^{\times k_{i,j}} \oplus \theta_{\mathrm{rank}(f_b)})[x(p \oplus f_b^{1/2})]^* - 
\theta_{\mathrm{rank}(a) m_j/m_i} \Vert.
\end{eqnarray*}
If we can show that $\theta_{\mathrm{rank}(a) m_j/m_i}$ is not Murray-von Neumann equivalent to a subprojection of $\xi^{\times k_{i,j}} \oplus 
\theta_{\mathrm{rank}(f_b)}$, then we will have that the last quantity above is $\geq 1/2$ (cf. \cite[Lemma 2.1]{To5}).  It will 
then follow that for every $j > i$ and every $x \in A_j$, 
\[
\Vert x \phi_{i,j}(b) x^* - \phi_{i,j}(a) \Vert \geq 1/2;
\]
in particular, $\langle a \rangle \nleq \langle b \rangle$, as desired.

By a straightforward adaptation of \cite[Lemma 2.1]{V1} (using the fact that the top Chern class of $\xi$ is not zero), $\theta_{\mathrm{rank}(a) m_j/m_i}$
will fail to be equivalent to a subprojection of $\xi^{\times k_{i,j}} \oplus \theta_{\mathrm{rank}(f_b)}$ if $\mathrm{rank}(a) m_j/m_i > 
\mathrm{rank}(f_b)$.  We have
\begin{eqnarray*}
\mathrm{rank}(f_b) - \mathrm{rank}(a)\cdot \frac{m_j}{m_i}& \leq & 2k l_{i,j} - \mathrm{rank}(a)\cdot \frac{m_j}{m_i} \\
& \leq & N_i \left( \frac{m_j}{m_i} - k_{i,j} \right) - \mathrm{rank}(a)\cdot \frac{m_j}{m_i} \\
& = & (N_i -\mathrm{rank}(a)) \cdot \frac{m_j}{m_i} - n_0 n_1 \cdots n_j, 
\end{eqnarray*}
so it will be enough to prove that 
\[
n_0 n_1 \cdots n_j > (N_i -\mathrm{rank}(a)) \cdot \frac{m_j}{m_i}.
\]
Rearranging and using the definitions of $m_i$ and $N_i$ we must show that
\[
\frac{(n_{i+1}+l_{i+1})\cdots(n_j + l_j)}{n_{i+1} \cdots n_j} \cdot (1- \mathrm{rank}(a)/N_i) < 1.
\]
Now $\mathrm{rank}(a) > (\eta/2)m_i$, so the right hand side above is less than 
\begin{equation}\label{ineq6}
\frac{(n_{i+1}+l_{i+1})\cdots(n_j + l_j)}{n_{i+1} \cdots n_j} \cdot \left(1- \frac{m_i \eta}{2 N_i} \right).
\end{equation}
The sequence $(m_i \eta)(2 N_i)$ is convergent to a nonzero limit, so for some $\gamma>0$, for all $i$
sufficiently large, the expression in (\ref{ineq6}) is strictly less than
\begin{equation}\label{ineq7}
\frac{(n_{i+1}+l_{i+1})\cdots(n_j + l_j)}{n_{i+1} \cdots n_j} \cdot (1- \gamma ).
\end{equation}
Increasing $i$ if necessary we may assume that 
\[
\frac{(n_{i+1}+l_{i+1})\cdots(n_j + l_j)}{n_{i+1} \cdots n_j} < \frac{1}{1-\gamma},
\]
whence the expression in (\ref{ineq7}) is strictly less than one, as required.
This completes the proof that $\mathrm{rc}(A^{(r)}=r$.  

Since each natural number divides some $m_i$ and each $X_i$ is contractible, we have
$\mathrm{K}_0(A^{(r)}) \cong \mathbb{Q}$, with the usual order structure and order unit.
The contractibility of $X_i$ also implies that $\mathrm{K}_1(A_i)=0$ for every $i$, whence
$\mathrm{K}_1(A^{(r)})=0$, too.  The pairing $\rho$ between traces and $\mathrm{K}_0$ is determined
uniquely since there is only one state on $\mathrm{K}_0(A^{(r)})$.  In order to complete
the proof that $\mathrm{Ell}(A^{(r)}) \cong \mathrm{Ell}(A^{(s)})$ for every $r,s \in \mathbb{R}^+ \backslash
\{0\}$, we must prove that $\mathrm{T}(A^{(r)}) \cong \mathrm{T}(A^{(s)})$.  

Recall that the tracial state space of $\mathrm{M}_k(\mathrm{C}(X))$ is homeomorphic to
the space $\mathcal{P}(X)$ of regular positive Borel probability measures on $X$.  Let 
$(A_i^{(r)},\phi_i)$ and $(A_i^{(s)},\psi_i)$ be inductive sequences as above, with simple limits
$A^{(r)}$ and $A^{(s)}$, respectively.  We have $\mathrm{Spec}(A_i^{(r)}) = [0,1]^{N_i}$ and
$\mathrm{Spec}(A_i^{(s)}) = [0,1]^{M_i}$.  Using the superscript $\sharp$ to denote the map induced on traces
by a $*$-homomorphism, we have
\[
\mathrm{T}(A^{(r)}) \cong \underset{\longleftarrow}{\lim} \ (\mathcal{P}([0,1]^{N_i},\phi_i^{\sharp}); \ \ \ 
\mathrm{T}(A^{(s)}) \cong \underset{\longleftarrow}{\lim} \ (\mathcal{P}([0,1]^{M_i},\psi_i^{\sharp}).
\]
We require sequences $(\gamma_i)$ and $(\delta_i)$ of continuous affine maps making the triangles in the diagram
\begin{equation}\label{inter}
\xymatrix{
{\mathcal{P}([0,1]^{N_1})} & 
{\mathcal{P}([0,1]^{N_2})}\ar[l]^{\phi_1^{\sharp}}\ar[dl]^{\delta_1} &
{\mathcal{P}([0,1]^{N_3})}\ar[l]^{\phi_2^{\sharp}}\ar[dl]^{\delta_2} & {\cdots}\ar[l]^{\phi_3^{\sharp}}\ar[dl]^{\delta_3} \\
{\mathcal{P}([0,1]^{M_1})}\ar[u]^{\gamma_{1}}&
{\mathcal{P}([0,1]^{M_2})}\ar[l]^{\psi_1^{\sharp}}\ar[u]^{\gamma_{2}} &
{\mathcal{P}([0,1]^{M_3})}\ar[l]^{\psi_2^{\sharp}}\ar[u]^{\gamma_{3}} & {\cdots}\ar[l]^{\psi_3^{\sharp}}
}
\end{equation}
commute ever more closely on ever larger finite sets as $i \to \infty$.  We will in fact be able to arrange
for near-commutation on the entire source space in each triangle. 

Let $\mu$ be a probability measure on $X^N$, and $K$ a subset of $\{1,\ldots,N\}$.  
We use $\mu_K$ to denote the measure on $X^{|K|}$ defined by 
integrating out those co-ordinates of $X^N$ not contained in $K$.   Straightforward calculation
shows that upon viewing $X_{i+1}$ as $X_i^{N_{i+1}/N_i}$ we have
\[
\phi_i^{\sharp}(\mu) =  \frac{n_{i+1}}{n_{i+1}+l_{i+1}} \left[ \frac{N_i}{N_{i+1}} \ \bigoplus_{l=1}^{N_{i+1}/N_i} \mu_{\{l\}} \right]
+  \frac{l_{i+1}}{n_{i+1}+l_{i+1}}  \lambda_i,
\]
where $\lambda_i$ is a convex combination of finitely many point masses.  A similar statement holds for $\psi_i^{\sharp}$. 
Since $l_{i+1}/(n_{i+1}+l_{i+1})$ is negligible for large $i$, we may in fact assume that 
\[
\phi_i^{\sharp}(\mu) =   \frac{N_i}{N_{i+1}} \ \bigoplus_{l=1}^{N_{i+1}/N_i} \mu_{\{l\}}; \ \ \ 
\psi_i^{\sharp}(\mu) =   \frac{M_i}{M_{i+1}} \ \bigoplus_{k=1}^{M_{i+1}/M_i} \mu_{\{k\}}
\]
for the purposes of our intertwining argument.  We may also assume, by compressing our sequences if necessary, that
$N_1 \ll M_1 \ll N_2 \ll M_2 \ll \cdots$.  Define
\[
\gamma_1(\mu) = \frac{1}{\lfloor M_1/N_1 \rfloor} \ \bigoplus_{l=1}^{\lfloor M_1/N_1 \rfloor} \mu_{\{(l-1)N_1+1,\ldots,lN_1\}}.
\]
Now set $B_k = \{(k-1)M_1+1,\ldots,kM_1\}$ for each $1 \leq k \leq \lfloor N_2/M_1 \rfloor$, and $D_t = 
\{(t-1)N_1,\ldots, t N_1\}$ for each $1 \leq t \leq N_2/N_1$.  Define
\[
\delta_1(\mu) = \frac{1}{\lfloor N_2/M_1 \rfloor} \ \bigoplus_{k=1}^{\lfloor N_2/M_1 \rfloor} \sigma_k^*(\mu_{B_k}),
\]
where $\sigma_k^*$ is the map induced on measures by the homeomorphism $\sigma_k:B_k \to B_k$ defined by the following property:
if $j$ is the first co-ordinate of $B_k$ contained in a $D_t$ which is itself contained in $B_k$, then $\sigma_k$ is the permutation
which subtracts $j-1 (\mathrm{mod} |B_k|)$ from each co-ordinate.  (The idea is that $\sigma_k$ moves all of the $D_t$s contained in 
$B_k$ ``to the beginning''.)

Let $L$ be the number of $D_t$s which are contained in some $B_k$.  Since $N_1 \ll M_1 \ll N_2$, we have that $(N_2 - N_1L)/N_2$ is 
(arbitrarily) small.  Now 
\[
\gamma_1 \circ \delta_1 (\mu) = \frac{1}{L} \bigoplus_{\{t \ | \ D_t \subseteq B_k, \ \mathrm{for \ some} \ k \}} \mu_{D_t},
\]
while
\[
\phi_1^{\sharp}(\mu) = \frac{1}{N_2/N_1} \bigoplus_{k=1}^{N_2/N_1} \mu_{D_k}.
\]
The difference $[(\gamma_1 \circ \delta_1)-\phi_1^{\sharp}](\mu)$ is a measure of total mass at most $2(N_2-N_1 L)/N_2$, and so the first triangle
from the diagram (\ref{inter}) commutes to within this tolerance on {\it all} of $\mathcal{P}([0,1]^{N_2})$.  The subsequent
$\gamma_i$s and $\delta_i$s are defined in a manner analogous to our definition of $\delta_1$, and this leads to the desired intertwining.
We conclude that $\mathrm{Ell}(A^{(r)}) \cong \mathrm{Ell}(A^{(s)})$, as desired.  

It remains to prove that if $A^{(r)}$ and $A^{(s)}$ are Morita equivalent, then $r/s \in \mathbb{Q}$.  
Suppose that they are so.  By the Brown-Green-Rieffel Theorem, $A^{(r)}$ and $A^{(s)}$ are stably
isomorphic, and so there are projections $p,q \in A^{(r)} \otimes \mathcal{K}$ such that 
$A^{(r)} \cong p (A^{(r)} \otimes \mathcal{K})p$ and $A^{(s)} \cong q(A^{(r)} \otimes \mathcal{K})q$.
Since $\mathrm{K}_0(A^{(r)} \otimes \mathcal{K}) = \mathrm{K}_0(A^{(r)}) = \mathbb{Q}$, there are 
natural numbers $n$ and $m$ such that $n[p] = m[q]$ in $\mathrm{K}_0$.  It is proved in \cite{V1}
that the construction used to arrive at $A^{(r)}$ and $A^{(s)}$ always produces C$^*$-algebras of
stable rank one, whence $A^{(r)} \otimes \mathcal{K}$ has stable rank one.  Thus, $\oplus_{i=1}^n \ p$ and
$\oplus_{j=1}^m \ q$ are Murray-von Neumann equivalent, and
\[
\mathrm{M}_n(A^{(r)}) \cong (\oplus_{i=1}^n \ p)(A^{(r)} \otimes \mathcal{K})(\oplus_{i=1}^n \ p)
\cong (\oplus_{i=1}^m \ q)(A^{(r)} \otimes \mathcal{K})(\oplus_{i=1}^m \ q) \cong \mathrm{M}_m(A^{(s)}).
\]
By \cite[Proposition 6.2 (ii)]{To3} we have 
\[
r/n = \mathrm{rc}\left( \mathrm{M}_n(A^{(r)}) \right) = \mathrm{rc} \left( \mathrm{M}_m(A^{(s)}) \right) = s/m,
\]
whence $r/s \in \mathbb{Q}$, as required.
\end{proof}

\begin{rems} {\rm If $r/s \notin \mathbb{Q}$, then $A^{(r)} \otimes \mathcal{K} \ncong A^{(s)} \otimes \mathcal{K}$.
This, to our knowledge, is the first example of simple separable amenable {\it stable} C$^*$-algebras which are not
isomorphic yet have the same Elliott invariant.}
\end{rems}

\end{document}